\documentclass[twoside]{article}

\usepackage[accepted]{aistats2019}

\usepackage[dvipsnames]{xcolor}
\usepackage[mathscr]{eucal}
\usepackage[noend]{algpseudocode}
\usepackage[subrefformat=parens]{subcaption}

\usepackage{algorithm}
\usepackage{amsfonts}
\usepackage{amsmath}
\usepackage{amssymb}
\usepackage{amsthm}
\usepackage{afterpage}
\usepackage{bm}
\usepackage{booktabs}
\usepackage{cancel}
\usepackage{comment}
\usepackage{fancyhdr}
\usepackage{grffile}
\usepackage{graphicx}
\usepackage{mathtools}
\usepackage{pdflscape}
\usepackage{tikz}
\usepackage{xifthen}
\usepackage{microtype}

\usepackage{color}

\usepackage[sort,colon,numbers]{natbib}
\renewcommand{\citet}{\citep}

\usepackage{hyperref}
\hypersetup{
  citecolor=blue!50!black,
  colorlinks=true,
  linkcolor=red!50!black,
  urlcolor=magenta,
}

\usetikzlibrary{bayesnet}
\usetikzlibrary{calc}
\usetikzlibrary{decorations.pathreplacing}
\usetikzlibrary{patterns}
\usetikzlibrary{positioning}
\usetikzlibrary{matrix}

\theoremstyle{plain}
\newtheorem{theorem}{Theorem}[section]
\newtheorem{lemma}[theorem]{Lemma}

\newtheorem{corollary}[theorem]{Corollary}
\newtheorem{proposition}[theorem]{Proposition}

\theoremstyle{definition}
\newtheorem{definition}[theorem]{Definition}
\newtheorem{problem}[theorem]{Problem}
\newtheorem{example}[theorem]{Example}
\newtheorem{remark}[theorem]{Remark}

\newcommand{\simiid}{\sim^{\mathrm{iid}}}

\newcommand{\dom}{\mathcal{T}}

\newcommand{\bp}{\mathbf{p}}
\newcommand{\bq}{\mathbf{q}}
\newcommand{\bh}{\mathbf{h}}

\newcommand{\tbp}{\tilde\bp}
\newcommand{\tbq}{\tilde\bq}

\newcommand{\abs}[1]{\lvert{#1}\rvert}
\newcommand{\card}[1]{\lvert{#1}\rvert}
\newcommand{\set}[1]{\left\lbrace{#1}\right\rbrace}

\newcommand{\EQ}{\textsf{EQ}}
\newcommand{\GT}{\textsf{GT}}
\newcommand{\LT}{\textsf{LT}}

\newcommand{\reject}{\textsf{reject}}
\newcommand{\notreject}{\textsf{not reject}}

\newcommand{\Hnull}{\mathsf{H}_0}
\newcommand{\Halt}{\mathsf{H}_1}

\newcommand{\pitmanyor}{\mathsf{CRP}}
\newcommand{\binompdf}[3]{
  \binom{#1}{#2}
  \left[#3\right]^{#1}
  \left[1-#3\right]^{#2-#1}}

\newcommand{\Rationals}{\mathbb{Q}}

\newcommand{\expect}[2][]{%
\ifthenelse{\isempty{#1}}
  {\mathbb{E}\left[{#2}\right]}
  {\mathbb{E}_{#1}\left[{#2}\right]}
}
\newcommand{\Prob}[2][]{#1{\Pr}\left\{#2\right\}}

\newcommand{\indicator}[1]{\mathbb{I}\left[{#1}\right]}
\newcommand{\indicatornm}[1]{\mathbb{I}[{#1}]}

\newcommand{\plt}{\prec} 
\newcommand{\pgt}{\succ} 
\newcommand{\vlt}{\vartriangleleft} 
\newcommand{\vgt}{\vartriangleright} 
\newcommand{\dictorder}{\vlt_{\mathrm d}} 

\newcommand{\Naturals}{\mathbb{N}}
\newcommand{\defas}{\coloneqq}
\newcommand{\br}{\mathbf{r}}
\newcommand{\tbr}{\tilde{\br}}
\newcommand{\given}{\,|\,}
\newcommand{\defn}[1]{{\bf{#1}}}

\allowdisplaybreaks

\begin{document}

%
\runningtitle{A Family of Exact Goodness-of-Fit Tests for
High-Dimensional Discrete Distributions}

%
\runningauthor{Saad, Freer, Ackerman, and Mansinghka}

\twocolumn[
  \aistatstitle{A Family of Exact Goodness-of-Fit Tests
    \\ for High-Dimensional Discrete Distributions}

  \aistatsauthor{
    Feras A.~Saad \And
    Cameron E.~Freer \And
    Nathanael L.~Ackerman \And
    Vikash K.~Mansinghka
  }
  \aistatsaddress{
    MIT \And
    MIT \And
    Harvard University \And
    MIT
  }
]


\begin{abstract}
The objective of goodness-of-fit testing is to assess whether a
dataset of observations is likely to have been drawn from a candidate probability
distribution.
This paper presents a rank-based family of goodness-of-fit tests that is specialized
to discrete distributions on high-dimensional domains.
The test is readily implemented using a simulation-based, linear-time
procedure.
The testing procedure can be customized by the practitioner using knowledge of
the underlying data domain.
Unlike most existing test statistics, the proposed test statistic is
distribution-free and its exact (non-asymptotic) sampling distribution is
known in closed form.
We establish consistency of the test against all alternatives by showing that the test statistic is
distributed as a discrete uniform if and only if the samples were drawn from the
candidate distribution.
We illustrate its efficacy for assessing the sample quality of approximate
sampling algorithms over combinatorially large spaces with intractable
probabilities, including random partitions in Dirichlet process mixture models
and random lattices in Ising models.
\end{abstract}


\section{Introduction}
\label{sec:introduction}

We address the problem of testing whether a dataset of observed samples was
drawn from a candidate probability distribution. This problem, known as
goodness-of-fit testing, is of fundamental interest and has applications in a
variety of fields including Bayesian statistics~\citep{gelman1996,talts2018},
high-energy physics~\citep{williams2010}, astronomy~\citep{peacock1983}, genetic
association studies~\citep{lewsi2012}, and psychometrics~\cite{andersen1973}.

Rank-based methods are a popular approach for assessing goodness-of-fit
and have received great attention in the nonparametric statistics
literature~\citep{erich1975nonparametrics}.
However, the majority of existing rank-based tests operate under the assumption
of continuous distributions~\citep[VI.8]{lehmann2005} and analogous methods for
discrete distributions that are theoretically rigorous, customizable using
domain knowledge, and practical to implement in a variety of settings
remain much less explored.

This paper presents a new connection between rank-based tests and discrete
distributions on high-dimensional data structures. By algorithmically specifying
an ordering on the data domain, the practitioner can quantitatively assess how
typical the observed samples are with respect to resampled data from the
candidate distribution. This ordering is leveraged by the test to effectively
surface distributional differences.

More specifically, we propose to test whether observations
$\set{y_1,\dots,y_n}$, taking values in a countable set $\dom$, were drawn from
a given discrete distribution $\bp$ on the basis of the rank of each $y_i$ with
respect to $m$ i.i.d.\ samples $\set{x_1,\dots,x_m}$ from $\bp$. If $y_i$ was
drawn from $\bp$ then we expect its rank to be uniformly distributed over
$\set{0,1,\dots,m}$. When the ranks show a deviation from uniformity, it is
unlikely that the $y_i$ were drawn from $\bp$. A key step is to use continuous
random variables to break any ties when computing the ranks. We call this
statistic the Stochastic Rank Statistic (SRS), which has several desirable
properties for goodness-of-fit testing:

\begin{enumerate}
\item The SRS is distribution-free: its sampling distribution under the
null does not depend on $\bp$. There is no need to construct
ad-hoc tables or use Monte Carlo simulation to estimate rejection regions.

\item The exact (non-asymptotic) sampling distribution of the SRS is a discrete
uniform. This exactness obviates the need to apply asymptotic approximations in
small-sample and sparse regimes.

\item The test is consistent against all alternatives. We show that the SRS is
distributed as a discrete uniform if and only if $\set{y_1,\dots,y_n}
\simiid \bp$.

\item The test gives the practitioner flexibility in deciding the set of
properties on which the observations be checked to agree with samples from
$\bp$. This flexibility arises from the design of the ordering on the
domain that is used to compute the ranks.

\item The test is readily implemented using a procedure that is linear-time in
the number of observations. The test is simulation-based and does not require
explicitly computing $\bp(x)$, which is especially useful for distributions with
intractable probabilities.
\end{enumerate}

While the test is consistent for any ordering $(\dom, \prec)$ over the domain
that is used to compute the SRS, the power of the test depends heavily on the
choice of $\prec$. We show how to construct orderings in a variety of domains by
(i) defining procedures that traverse and compare discrete data structures; (ii)
composing probe statistics that summarize key numerical characteristics; and
(iii) using randomization to generate arbitrary orderings.

The remainder of the paper is organized as follows.
Section~\ref{sec:goodness-of-fit-problem} reviews the goodness-of-fit problem
and discusses related work.
Section~\ref{sec:srs-statistic} presents the proposed test and
several theoretical properties.
Section~\ref{sec:example} gives conceptual examples for distributions over
integers, binary strings, and partitions.
Section~\ref{sec:applications} applies the method to (i) compare approximate
Bayesian inference algorithms over mixture assignments in a Dirichlet process
mixture model and (ii) assess the sample quality of random lattices from approximate
samplers for the Ising model.


\begin{figure*}[ht]
\begin{tikzpicture}

\node[
  rectangle,
  draw = none,
  align = center,
  inner sep = 0pt,
] (outcome-space)  at (0,0) {
  \includegraphics[width=.125\linewidth]{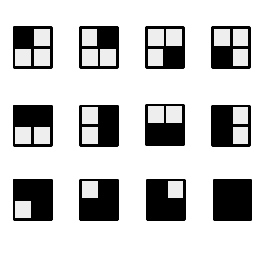}
};

\node [
  rectangle,
  draw = none,
  align = center,
  text width=3cm,
  inner sep = 0pt,
  above = .1 of outcome-space,
] (outcome-space-label){
  \baselineskip=0pt\footnotesize
  observation domain \\ (finite/countable) \par
};

\node[
  rectangle,
  draw = black,
  align = center,
  below = 0.5 of outcome-space,
  text width = .5cm,
  minimum height = .75cm,
] (dist-q) {
  $\bq$
};

\node [
  rectangle,
  draw = none,
  align = center,
  text width=3cm,
  inner sep = 0pt,
  below = .1 of dist-q.south,
  anchor = north,
] (dist-q-label) {
  \baselineskip=0pt\footnotesize
  unknown \\ distribution \par
};

\matrix[
  matrix of nodes,
  inner sep=2pt,
  nodes in empty cells,
  right = 1. of dist-q.east,
  anchor = west,
  nodes={rectangle, text centered, text width=1em, align=center, text height=2.5ex},
] (samples) {
  \includegraphics[width=\textwidth]{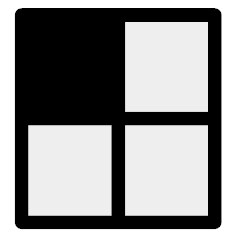}
  & \raisebox{3pt}{\tiny{vs}}
  & \includegraphics[width=\textwidth]{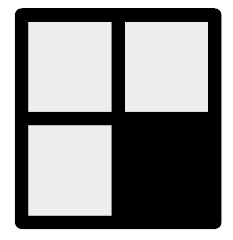}
  & \includegraphics[width=\textwidth]{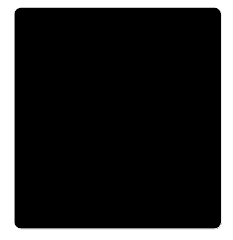}
  & $\dotsm$
  & \includegraphics[width=\textwidth]{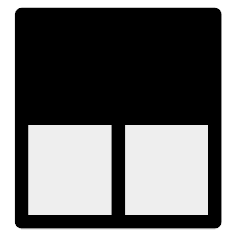}
  \\
  \includegraphics[width=\textwidth]{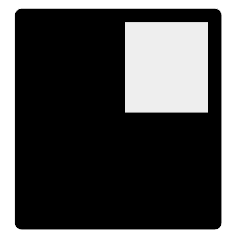}
  & \raisebox{3pt}{\tiny{vs}}
  & \includegraphics[width=\textwidth]{assets/schematic/gof-outcomes-icon-12.png}
  & \includegraphics[width=\textwidth]{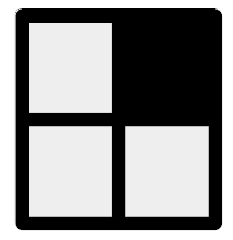}
  & $\dotsm$
  & \includegraphics[width=\textwidth]{assets/schematic/gof-outcomes-icon-2.png}
  \\
  \vdots & & \vdots & \vdots &  & \vdots\\
  \includegraphics[width=\textwidth]{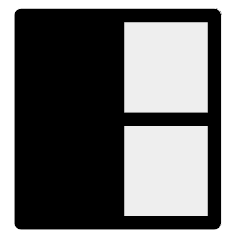}
  & \raisebox{3pt}{\tiny{vs}}
  & \includegraphics[width=\textwidth]{assets/schematic/gof-outcomes-icon-12.png}
  & \includegraphics[width=\textwidth]{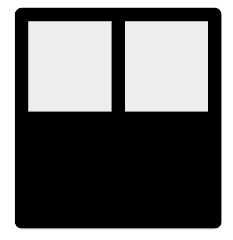}
  & $\dotsm$
  & \includegraphics[width=\textwidth]{assets/schematic/gof-outcomes-icon-3.png}
  \\
};

\node[inner sep=0, left=0 of samples-1-1] (y1) {\scriptsize $y_1$};
\node[inner sep=0, left=0 of samples-2-1] (y2) {\scriptsize $y_2$};
\node[inner sep=0, left=0 of samples-4-1] (yn) {\scriptsize $y_n$};

\node[above=0 of samples-1-3] (x1) {\scriptsize $X_1$};
\node[above=0 of samples-1-4] (x2) {\scriptsize $X_2$};
\node[above=0 of samples-1-6] (xn) {\scriptsize $X_m$};

\node[
  rectangle,
  draw = black,
  align = center,
  text width = .5cm,
  minimum height = .75cm,
] (dist-p) at (outcome-space.25 -| samples-1-5.north west) {
  $\bp$
};

\node [
  rectangle,
  draw = none,
  align = center,
  text width=3cm,
  inner sep = 0pt,
] (dist-p-label) at (outcome-space-label -| dist-p) {
  \baselineskip=0pt\footnotesize
  hypothesized \\ distribution \par
};

\draw[line width=.5pt,-stealth] (outcome-space.south) -- (dist-q.north);
\draw[line width=.5pt,-stealth] (outcome-space.25) -- (dist-p.west);

\draw[decorate, decoration={brace,amplitude=5pt}]
  ([xshift=-.325cm]samples-4-1.south west) -- ([xshift=-.325cm]samples-1-1.north west)
  node[pos=0.5, xshift=-.2cm, anchor=east]{\scriptsize{iid}};

\draw[decorate, decoration={brace,amplitude=5pt}]
  ([yshift=.45cm]samples-1-3.north west) -- ([yshift=.45cm]samples-1-6.north east)
  node[pos=0.5, yshift=.2cm, anchor=south]{\scriptsize{iid}};

\draw (samples-1-1.south west) -- (samples-1-6.south east);
\draw (samples-2-1.south west) -- (samples-2-6.south east);
\draw (samples-3-1.south west) -- (samples-3-6.south east);
\draw (samples-4-1.south west) -- (samples-4-6.south east);

\node[
  draw = black,
  inner sep = 4pt,
  minimum height = .75cm,
  right = .5 of dist-p.east,
  anchor = west,
] (linear-ordering) {
  \includegraphics[width=1em]{assets/schematic/gof-outcomes-icon-1.png}
  \raisebox{3pt}{${\prec}$}
  \includegraphics[width=1em]{assets/schematic/gof-outcomes-icon-2.png}
  \raisebox{3pt}{${\prec}{\dotsm}{\prec}$}
  \includegraphics[width=1em]{assets/schematic/gof-outcomes-icon-5.png}
  \raisebox{3pt}{${\prec}{\dotsm}{\prec}$}
  \includegraphics[width=1em]{assets/schematic/gof-outcomes-icon-11.png}
  \raisebox{3pt}{${\prec}$}
  \includegraphics[width=1em]{assets/schematic/gof-outcomes-icon-12.png}
  \\
};

\node [
  rectangle,
  draw = none,
  align = center,
  text width=4cm,
  inner sep = 0pt,
] (dist-p-label) at (dist-p-label -| linear-ordering) {
  \baselineskip=0pt\footnotesize
  linear ordering $\prec$ on domain \par
};

\node[
  rectangle,
  draw = black,
  inner sep = 0pt,
  right = .5 of samples.east,
  anchor = west,
  text width = 1.75cm,
  minimum height = 2cm,
  minimum width = 1.75cm,
  align = center,
] (ranking-procedure) {
  \footnotesize
  stochastic \\ ranking \\ procedure \par
};

\draw[-stealth] (ranking-procedure.north |- linear-ordering.south) -- (ranking-procedure.north);
\draw[-stealth] (samples-1-6.east) -- (ranking-procedure.west |- samples-1-6.east);
\draw[-stealth] (samples-2-6.east) -- (ranking-procedure.west |- samples-2-6.east);
\draw[-stealth] (samples-3-6.east) -- (ranking-procedure.west |- samples-3-6.east);
\draw[-stealth] (samples-4-6.east) -- (ranking-procedure.west |- samples-4-6.east);

\node [
  rectangle,
  draw = none,
  right= 2.7 of samples-1-6.east,
  text width = 2.8cm,
  anchor = west,
  inner xsep = 2pt,
  inner ysep = 0pt,
] (rank-1) {
  \scriptsize $r_n\,{=}\,\mathrm{Rank}(y_1, X_{1,1:m})$
};

\node [
  rectangle,
  draw = none,
  right= 2.7 of samples-2-6.east,
  text width = 2.8cm,
  anchor = west,
  inner xsep = 2pt,
  inner ysep = 0pt,
] (rank-2) {
  \scriptsize $r_n\,{=}\,\mathrm{Rank}(y_1, X_{2,1:m})$
};

\node [
  rectangle,  rectangle,
  draw = none,
  right= 2.7 of samples-3-6.east,
  text width = 2.8cm,
  anchor = west,
  inner xsep = 2pt,
  inner ysep = 0pt,= 0pt,
] (rank-3) {
  \scriptsize \,$\dots$
};

\node [
  rectangle,
  draw = none,
  right= 2.7 of samples-4-6.east,
  text width = 2.8cm,
  inner xsep = 2pt,
  inner ysep = 0pt,
  anchor = west,
] (rank-4) {
  \scriptsize $r_n\,{=}\,\mathrm{Rank}(y_1, X_{n,1:m})$
};

\draw[-stealth] (ranking-procedure.east |- rank-1.west) -- (rank-1.west);
\draw[-stealth] (ranking-procedure.east |- rank-2.west) -- (rank-2.west);
\draw[-stealth] (ranking-procedure.east |- rank-3.west) -- (rank-3.west);
\draw[-stealth] (ranking-procedure.east |- rank-4.west) -- (rank-4.west);

\node[
  right = 2.25 of linear-ordering,
  rectangle,
  draw = black,
  minimum height = .75cm,
] (histogram) {
  \scriptsize histogramming
};

\coordinate (anchor) at ([xshift=.5cm]linear-ordering.east);
\draw[-stealth] (rank-1.east) -| (anchor) -- (histogram.west);
\draw[-stealth] (rank-2.east) -| (anchor) -- (histogram.west);
\draw[-stealth] (rank-3.east) -| (anchor) -- (histogram.west);
\draw[-stealth] (rank-4.east) -| (anchor) -- (histogram.west);

\node[
  rectangle,
  inner sep = 0pt,
  minimum height = 1cm,
  minimum width = 1.5cm,
  text width = 2.75cm,
  below left = 1.5 and 1.3 of histogram.south,
  anchor = north
] (hist-yes) {
  \includegraphics[width=\textwidth]{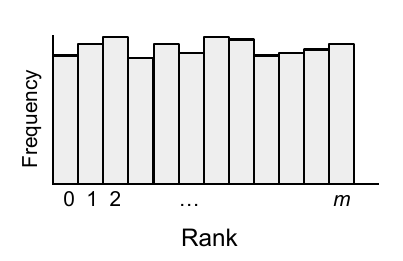}
};

\node[
  rectangle,
  inner sep = 0pt,
  minimum height = 1cm,
  minimum width = 1.5cm,
  text width = 2.75cm,
  below right = 1.5 and 1.3 of histogram.south,
  anchor = north
] (hist-no) {
  \includegraphics[width=\textwidth]{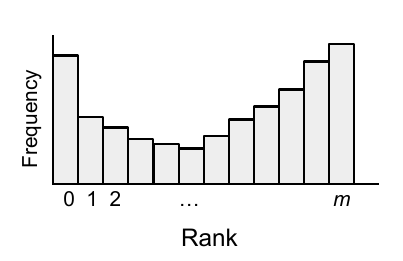}
};

\node [
  rectangle,
  draw = black,
  minimum height = .6cm,
  align = center,
  text width = 2cm,
  inner sep = 2pt,
  above = 0 of hist-yes.north,
  anchor = south,
] (hist-yes-label) {
  \baselineskip=0pt\footnotesize
  \scriptsize $\bp\,{\approx}\,\bq$ \\
  uniform
  \par
};

\node [
  rectangle,
  draw = black,
  minimum height = .6cm,
  align = center,
  text width = 2cm,
  inner sep = 2pt,
  above = 0 of hist-no.north,
  anchor = south,
] (hist-no-label) {
  \baselineskip=0pt\footnotesize
  \scriptsize $\bp\,{\not\approx}\,\bq$ \\
  non-uniform
  \par
};

\draw[-stealth] (histogram.south) -- (hist-yes-label.north);
\draw[-stealth] (histogram.south) -- (hist-no-label.north);

\node[
  draw = black,
  below = 0.5 of dist-q-label.south west,
  text width = 3.75cm,
  inner xsep = 0pt,
  inner ysep = 1pt,
  anchor = north west,
  align = center,
] (step-1) {\baselineskip=0pt\footnotesize
  \tiny \textbf{Stage 1}
};

\node[
  draw = black,
  right = 0 of step-1.east,
  text width = 2.7cm,
  inner xsep = 0pt,
  inner ysep = 1pt,
  anchor = west,
  align = center,
] (step-2)  {\baselineskip=0pt\footnotesize
  \tiny \textbf{Stage 2}
};

\node[
  draw = black,
  right = 0 of step-2.east,
  text width = 5.25cm,
  inner xsep = 0pt,
  inner ysep = 1pt,
  anchor = west,
  align = center,
] (step-3)  {\baselineskip=0pt\footnotesize
  \tiny \textbf{Stage 3}
};

\node[
  draw = black,
  right = 0 of step-3.east,
  text width = 5.3cm,
  inner xsep = 0pt,
  inner ysep = 1pt,
  anchor = west,
  align = center,
] (step-4)  {\baselineskip=0pt\footnotesize
  \tiny \textbf{Stage 4}
};

\end{tikzpicture}

\caption{Overview of the proposed goodness-of-fit test for discrete distributions.
\textbf{Stage 1}: Observations $\set{y_1,\dots,y_n}$ are assumed to be drawn
i.i.d.\ from an unknown discrete distribution $\bq$ over a finite or countable
observation domain $\dom$ (shown in the top-left corner).
\textbf{Stage 2}: For each $y_i$, $m$ samples $\set{X_{i1},\dots,X_{im}}$ are simulated
i.i.d.\ from the candidate distribution $\bp$ over $\dom$.
\textbf{Stage 3}: Given a total order $\prec$ on $\dom$ and the observed and
simulated data, a stochastic ranking procedure returns the rank $r_i$
of each $y_i$ within $\set{X_{i1},\dots,X_{im}}$, using uniform random numbers
to ensure the ranks are unique.
\textbf{Stage 4}: The histogram of the ranks $\set{r_1,\dots,r_n}$ is analyzed
for uniformity over $\set{0,1,\dots,m}$.}
\label{fig:workflow}
\end{figure*}


\section{The Goodness-of-Fit Problem}
\label{sec:goodness-of-fit-problem}

\begin{problem}
Let $\bp$ be a candidate discrete distribution over a finite or countably
infinite domain $\dom$. Given observations $\set{y_1,\dots,y_n}$ drawn i.i.d.\
from an unknown distribution $\bq$ over $\dom$, is there sufficient evidence to
reject the hypothesis $\bp = \bq$?
\end{problem}

In the parlance of statistical testing, we have the following null and
alternative hypotheses:
\begin{align*}
\Hnull \defas [\bp = \bq] && \Halt \defas [\bp \ne \bq].
\end{align*}
A statistical test $\phi_n\colon \dom^{n} \to \set{\reject, \notreject}$ says,
for each size $n$ dataset, whether to \textsf{reject} or \textsf{not reject} the
null hypothesis $\Hnull$. We define the significance level
\begin{equation}
\alpha \defas \Prob{\phi_n(Y_{1:n}) = \reject \mid \Hnull}
\end{equation}
to be the probability of incorrectly declaring $\reject$. For a given level
$\alpha$, the performance of the test $\phi_n$ is characterized by its power
\begin{equation}
\beta \defas \Prob{\phi_n(Y_{1:n}) = \reject \mid \Halt},
\end{equation}
which is the probability of correctly declaring \reject.

Classical goodness-of-fit tests for nominal (unordered) data include the
multinomial test~\citep{horn1977};
Pearson chi-square test~\citep{pearson1900};
likelihood-ratio test~\citep{williams1976};
nominal Kolmogorov--Smirnov test~\citep{hoeffding1965,pettitt1977};
and power-divergence statistics~\citep{read1998}.
For ordinal data, goodness-of-fit test statistics include the ordinal Watson,
Cram\'er--von Mises, and Anderson--Darling~\citep{choulakian1994} tests as well as
the ordinal Kolmogorov--Smirnov~\citep{conover1972,arnold2011}.
These approaches typically suffer from statistical issues in large domains. They
assume that $\bp(x)$ is easy to evaluate (which is rarely possible in modern
machine-learning applications such as graphical models) and/or require that each
discrete outcome $x\,{\in}\,\dom$ has a non-negligible expectation $n\bp(x)$
\citep{olivares2010,starnes2010} (which requires a large number of observations
$n$ even when $\bp$ and $\bq$ are noticeably far from one another).
In addition, the rejection regions of these statistics are either
distribution-dependent (which requires reestimating the region for each new
candidate distribution $\bp$) or asymptotically distribution-free (which is
inexact for finite-sample data and imposes additional statistical assumptions on
$\bp$ and $\bq$).
The Mann--Whitney U~\citep{mann1947}, which is also a rank-based test that bears
some similarity to the SRS, is only consistent under median shift,
whereas the proposed method is consistent under general distributional inequality.

Recent work in the theoretical computer science literature has established
computational and sample complexity bounds for testing approximate equality of
discrete distributions \citet{batu2000}.
These methods have been primarily studied from a theoretical perspective and
have not been shown to yield practical goodness-of-fit tests in practice, nor
have they attained widespread adoption in the applied statistics community.
For instance, the test in \citet{acharya2015} is based on a variant
of Pearson chi-square. It requires enumerating over the domain
$\dom$ and representing $\bp(x)$ explicitly. The test in \citet{valiant2011}
requires specifying and solving a complex linear program.
While these algorithms may obtain asymptotically sample-optimal limits, they are
designed to detect differences between $\bp$ and $\bq$ in a way that is robust
to highly adversarial settings. These tests do not account for any structure in
the domain $\dom$ that can be leveraged by the practitioner to effectively
surface distributional differences.

Permutation and bootstrap resampling of test statistics are another family of
tests for goodness-of-fit \citep{good2004}. Theoretically rigorous and
consistent tests can be obtained using kernel methods, including the maximum
mean discrepancy~\citep{gretton2012} and discrete Stein
discrepancy~\citep{yang2018}.
Since the null distribution is unknown, rejection regions are estimated
by bootstrap resampling, which may be inexact due to discreteness
of the data.
Instead of bootstrapping, the SRS can be used to obtain an exact,
distribution-free test by defining an ordering using the kernel. This
connection is left for future work.


\section{A Family of Exact and Distribution-Free GOF Tests}
\label{sec:srs-statistic}

In this section we describe our proposed method for addressing the
goodness-of-fit problem. The proposed procedure combines (i) the intuition from
existing methods for ordinal data~\citep{choulakian1994} that the
deviation between the expected CDF and empirical CDF of the sample serves as a
good signal for goodness-of-fit, with (ii) the flexibility of probe statistics
in Monte Carlo-based resampling tests~\citep{good2004}
to define, using an ordering $\prec$ on $\dom$, characteristics of the
distribution that are of interest to the experimenter.
Figure~\ref{fig:workflow} shows the step-by-step workflow
of the proposed test and Algorithm~\ref{alg:gof-uniform-rank} formally
describes the testing procedure.


\begin{algorithm}[ht]
\footnotesize
\caption{Exact GOF Test using SRS}
\label{alg:gof-uniform-rank}
\begin{algorithmic}[1]
\Require{
  $\left\{ \begin{tabular}{l}
  simulator for candidate dist. $\bp$ over $\dom$; \\
  i.i.d.\ samples $\set{y_1,y_2,\dots,y_n}$ from dist. $\bq$;\\
  strict total order $\prec$ on $T$, of any order type; \\
  number $m \ge 1$ of datasets to resimulate; \\
  significance level $\alpha$ of hypothesis test; \\
  \end{tabular}
  \right.$
}
\Ensure{Decision to reject the null hypothesis
  \mbox{$\Hnull\,{:}\,\bp\,{=}\,\bq$}
  versus alternative hypothesis $\Halt\,{:}\,\bp\,{\ne}\,\bq$ at level $\alpha$.}
\For{$i=1,2,\dots,n$}
  \State $X_1, X_2, \dots, X_m \simiid \bp$
    \label{algline:gof-uniform-simulate-obs}
  \State $U_0, U_1, \dots, U_m \simiid \mathsf{Uniform}(0,1)$
    \label{algline:gof-uniform-rank-uniforms}
  \State $r_{i} \gets \sum_{k=1}^m \indicatornm{X_k \prec y_i}
    +\indicatornm{X_k = y_i, U_k < U_0}$
    \label{algline:gof-uniform-break-ties}
\EndFor
\State Use a standard hypothesis test to compute
  $p$-value of $\set{r_{1},\dots,r_{n}}$ under a discrete uniform
  on $\set{0,\dots,m}$.
  \label{algline:gof-uniform-pval}
\State \Return \textsf{reject} if $p \le \alpha$, else \textsf{not reject}.
\end{algorithmic}
\end{algorithm}

The proposed method addresses shortcomings of existing statistics in sparse
regimes. It does not require the ability to compute $\bp(x)$ and it is not based
on comparing the expected frequency of each $x\,{\in}\,\dom$ (which is often
vanishingly small) with its observed frequency. Furthermore, the stochastic rank
statistics $r_i$ have an exact and distribution-free sampling distribution. The
following theorem establishes that the $r_i$ are uniformly distributed if and only if
$\bp\,{=}\,\bq$. (Proofs are in the Appendix.)

\vspace*{5pt}
\begin{theorem}
\label{thm:uniformity}
Let $\dom$ be a finite or countably infinite set, let $\prec$ be a strict total
order on $\dom$, let $\bp$ and $\bq$ be two probability distributions on $\dom$,
and let $m$ be a positive integer. Consider the following random variables:
\begin{align}
X_0 &\sim \bq
  \label{eq:proc-sim-xq} \\
X_1, X_2, \dots, X_m &\simiid \bp
  \label{eq:proc-sim-xp} \\
U_0, U_1, U_2, \dots, U_m &\simiid \mathsf{Uniform}(0,1)
  \label{eq:proc-sim-up} \\
R =
  \textstyle\sum_{j=1}^m \indicator{X_j \prec X_0}
  &+ \indicator{X_j = X_0, U_j < U_0}.
  \label{eq:proc-stochastic-rank}
\end{align}
Then $\bp=\bq$ if and only if for all $m\,{\ge}\,1$, the rank $R$ is distributed
as a discrete uniform random variable on the set of integers
\mbox{$[m+1]\,{\defas}\,\set{0,1,\dots,m}$}.
\end{theorem}

Note that the $r_i$ in line~\ref{algline:gof-uniform-break-ties} of
Algorithm~\ref{alg:gof-uniform-rank} are $n$ i.i.d.\ samples of the random
variable $R$ in Eq.~\eqref{eq:proc-stochastic-rank}, which is the rank of
$X_0\,{\sim}\,\bq$ within a size $m$ sample $X_{1:m}\,{\simiid}\,\bp$. For
Theorem~\ref{thm:uniformity}, it is essential that ties are broken by pairing
each $X_i$ with a uniform random variable $U_i$, as opposed to, e.g., breaking
each tie independently with probability $1/2$, as demonstrated by the next
example.

\begin{example}
\label{example:uniform-bernoulli-tie-break}
Let $\dom$ contain a single element. Then all the $X_i$ (for $0 \le i \le m$) are
equal almost surely. Break each tie between $X_0$ and $X_j$ by flipping a fair
coin. Then $R$ is binomially distributed with $m$ trials and weight $1/2$, not
uniformly distributed over $[m+1]$.
\end{example}

We now establish theoretical properties of $R$ which form the basis of the
goodness-of-fit test in Algorithm~\ref{alg:gof-uniform-rank}. First note that in
the case where all the $X_i$ are almost surely distinct, the forward direction
of Theorem~\ref{thm:uniformity}, which establishes that if $\bp = \bq$ then the
rank $R$ is uniform for all $m \ge 1$, is easy to show and is known in the
statistical literature \citep{ahsanulla2013}. However no existing results
make the connection between rank statistics and discrete random variables over
countable domains with ties broken stochastically. Nor do they establish that
$\bp\,{=}\,\bq$ is a \textit{necessary} condition for uniformity of $R$ (across all
$m$ beyond some integer) and can therefore be used as the basis of a consistent
goodness-of-fit test. We now state an immediate consequence of
Theorem~\ref{thm:uniformity}.

\begin{corollary}
\label{corr:p-ne-q-exists}
If $\bp \ne \bq$, then there is some $M \ge 1$ such that $R$ is not uniformly
distributed on $[M+1]$.
\end{corollary}

The next theorem significantly strengthens Corollary~\ref{corr:p-ne-q-exists} by
showing that if $\bp \ne \bq$, the rank statistic is non-uniform for
\textit{all but finitely many} $m$.

\begin{theorem}
\label{thm:p-ne-q-finitely-many-m}
Let $\bp \ne \bq$ and $M$ be defined as in Corollary~\ref{corr:p-ne-q-exists}.
Then for all $m \ge M$, the rank $R$ is not uniformly
distributed on $[m+1]$.
\end{theorem}

In fact, unless $\bp$ and $\bq$ satisfy an adversarial symmetry relationship
under the selected ordering $\prec$, the rank is non-uniform for
\textit{all} $m \ge 1$.

\begin{corollary}
\label{corr:total-order-no-symmetry}
Let $\vlt$ denote the lexicographic order on $\dom \times [0,1]$ induced by
$(\dom, \prec)$ and $([0,1], <)$. Suppose $\Prob{(X,U_1) \vlt (Y,U_0)} \ne 1/2$
for $Y \sim \bq$, $X \sim \bp$, and $U_0, U_1 \simiid \mathsf{Uniform}(0,1)$.
Then for all $m \ge 1$, the rank $R$ is not uniformly distributed on $[m+1]$.
\end{corollary}

The next theorem establishes the existence of an ordering on $\dom$ satisfying
the hypothesis of Corollary~\ref{corr:total-order-no-symmetry}.

\begin{theorem}
\label{thm:non-uniform-all-m} If $\bp \ne \bq$, then there is an ordering
$\prec^*$ whose associated rank statistic $R$ is non-uniform for $m = 1$ (and
hence by Theorem~\ref{thm:p-ne-q-finitely-many-m} for all $m \ge 1$).
\end{theorem}

Intuitively, $\prec^*$ sets elements $x\in \dom$ which have a high probability
under $\bq$ to be ``small'' in the linear order, and elements $x\,{\in}\,\dom$
which have a high probability under $\bp$ to be ``large'' in the linear order.
More precisely, $\prec^*$ maximizes the sup-norm distance between the induced
cumulative distribution functions $\tbp$ and $\tbq$ of $\bp$ and $\bq$,
respectively (Figure~\ref{fig:gof-bit-strings}). Under a slight variant of this
ordering, for finite $\dom$, the next theorem establishes the sample complexity
required to obtain exponentially high power in terms of the statistical distance
$L_\infty(\bp, \bq) = \sup_{x\in\dom}\abs{\bp(x) - \bq(x)}$ between $\bp$ and
$\bq$.

\begin{theorem}
\label{thm:sample-complexity}
Given significance level $\alpha=2\Phi(-c)$ for $c>0$, there is
an ordering for which the proposed test
with $m=1$ achieves power $\beta \ge 1 - \Phi(-c)$ using
\begin{equation}
\label{eq:sample-complexity}
n \approx 4c^2/L_\infty(\bp, \bq)^4
\end{equation}
samples from $\bq$, where $\Phi$ is the cumulative distribution function of a
standard normal.
\end{theorem}

This key result is independent of the domain size and establishes a lower bound
for any $\prec$ because it is based on the optimal ordering $\prec^*$.
The next theorem derives the exact sampling distribution for any pair of
distributions $(\bp, \bq)$, which is useful for simulation studies
(e.g., Figure~\ref{fig:gof-bit-strings}) that characterize the power of the SRS.

\begin{theorem}
\label{thm:dist-r0-combinatorial}
The distribution of $R$ is given by
\begin{align}
\Prob{R = r} = \sum_{x \in \dom}H(x, m, r) \,\bq(x)
  \label{eq:dist-r0-combinatorial}
\end{align}
for $0 \le r \le m$,
where $H(x,m,r) \defas$
\begin{align*}
\;\begin{cases}
  \displaystyle\sum_{e=0}^{m} \Bigg\{
    \Bigg[ \sum_{j=0}^{e}
    \displaystyle\binom{m-e}{r-j}
      \left[ \frac{\tilde\bp(x)}{1-\bp(x)} \right]^{r-j}
      \span \notag \\
    \quad \displaystyle\left[ 1-\frac{\tilde\bp(x)}{1-\bp(x)} \right]^{(m-e)-(r-j)}
      \left(\frac{1}{e+1}\right) \Bigg] \span \notag \\
    \quad \displaystyle\binompdf{m}{e}{\bp(x)} \Bigg\}
    & \mathrm{if}\, 0\,{<}\,\bp(x)\,{<}\,1, \\
  \displaystyle\binompdf{r}{m}{\tilde\bp(x)}
    & \mathrm{if}\, \bp(x)\,{=}\,0, \\[10pt]
  \displaystyle\frac{1}{m+1}
    & \mathrm{if}\, \bp(x)\,{=}\,1, \notag \\[10pt]
  \end{cases}
\end{align*}
and $\tilde\bp(x) \defas \sum_{x' \prec x}\bp(x)$ is the CDF of $\bp$.
\end{theorem}


\section{Examples}
\label{sec:example}


\begin{figure}[!b]
\centering
\vspace{-.25cm}
\includegraphics[width=.975\linewidth]{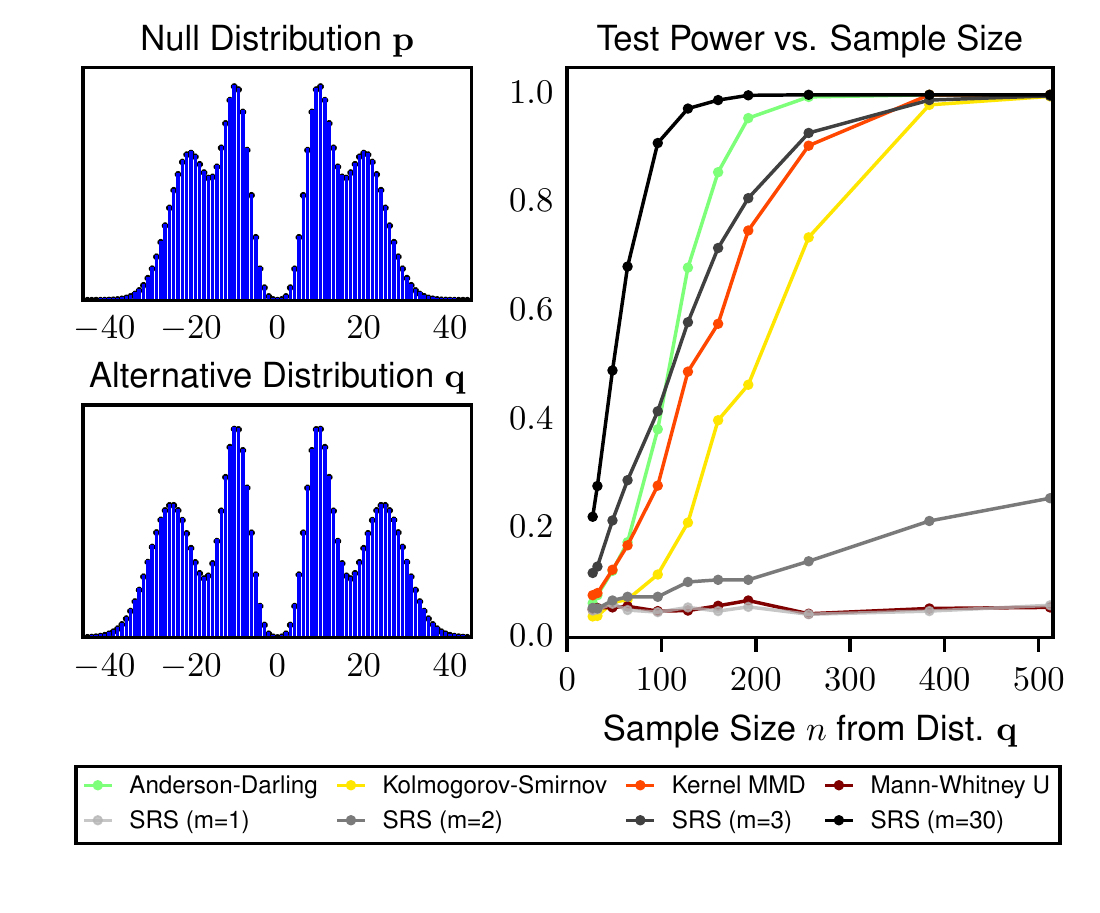}
\captionsetup{skip=0pt}
\caption{The left panel shows a pair $(\bp, \bq)$ of reflected, bimodal Poisson
distributions with slight location shift. The right plot compares the power of
testing $\bp\,{=}\,\bq$ using the SRS (for various choices of $m$) to several
baseline methods.}
\label{fig:poisson}
\end{figure}


\begin{figure*}[!htb]



\begin{subfigure}{\linewidth}
\begin{subfigure}[b]{.325\linewidth}
\includegraphics[width=\linewidth]{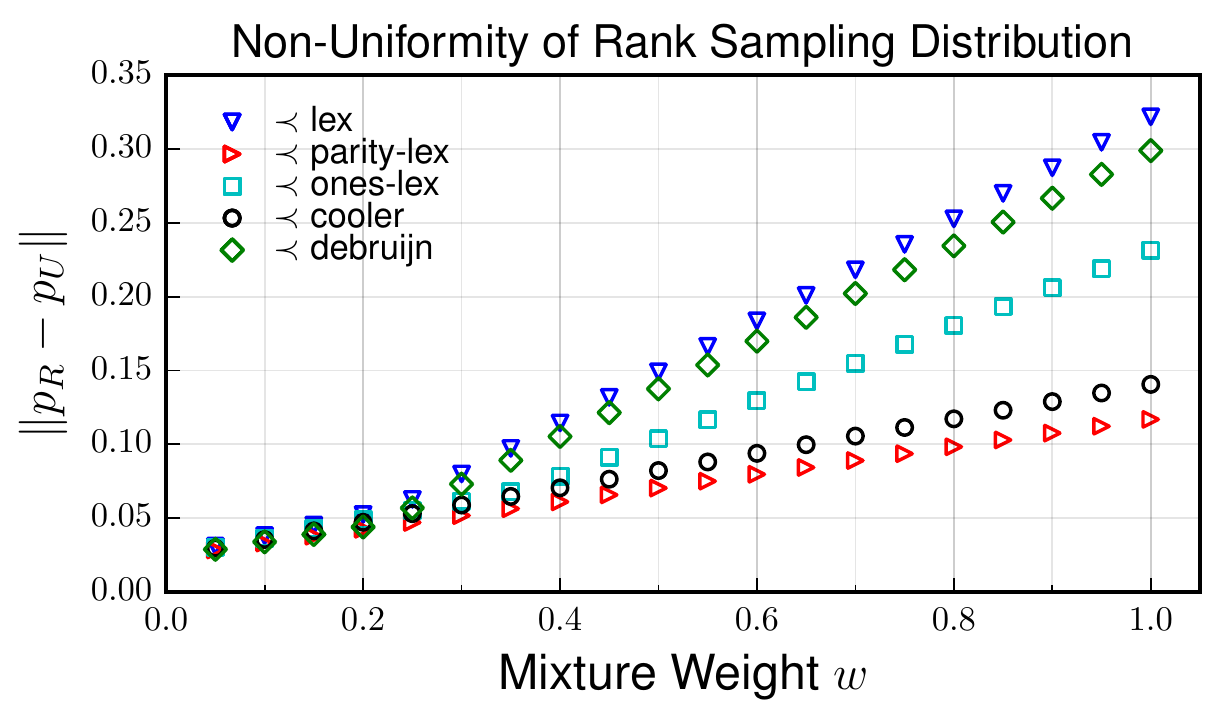}%
\end{subfigure}%
\begin{subfigure}[b]{.675\linewidth}
\includegraphics[width=\linewidth]{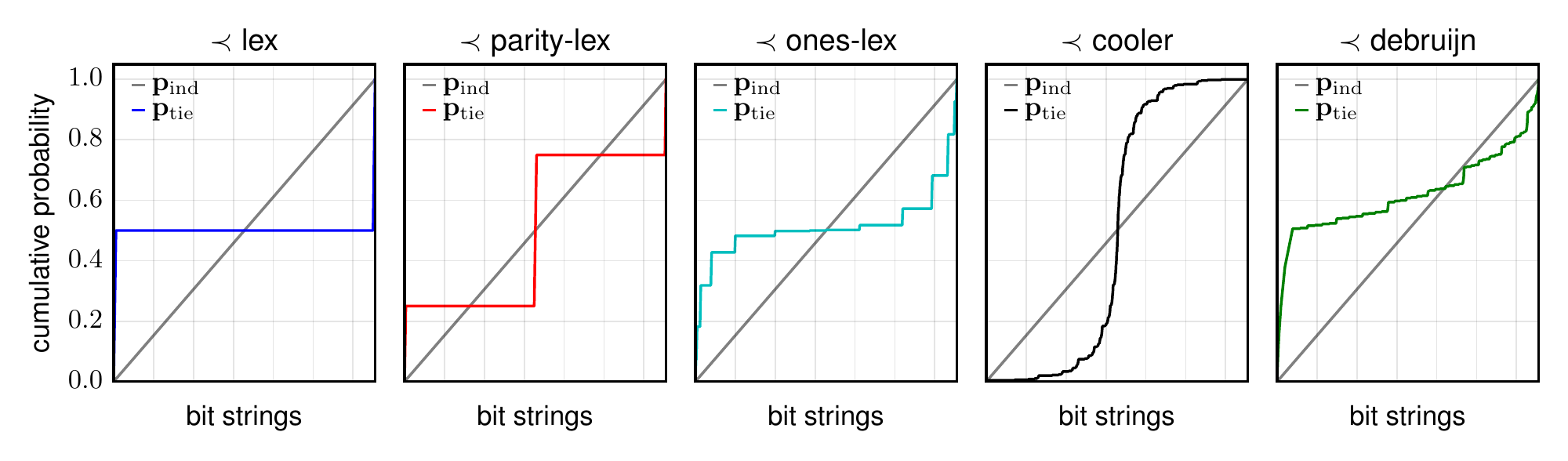}\hfill%
\end{subfigure}
\captionsetup{skip=0pt, belowskip=5pt, font=footnotesize}
\subcaption{$\bp \defas \bp_{\rm ind} \quad\quad \bq \defas w \bp_{\rm tie} + (1-w)\bp_{\rm ind}$}
\label{subfig:gof-bit-strings-tied}
\end{subfigure}

\begin{subfigure}{\linewidth}
\begin{subfigure}[b]{.325\linewidth}
\includegraphics[width=\linewidth]{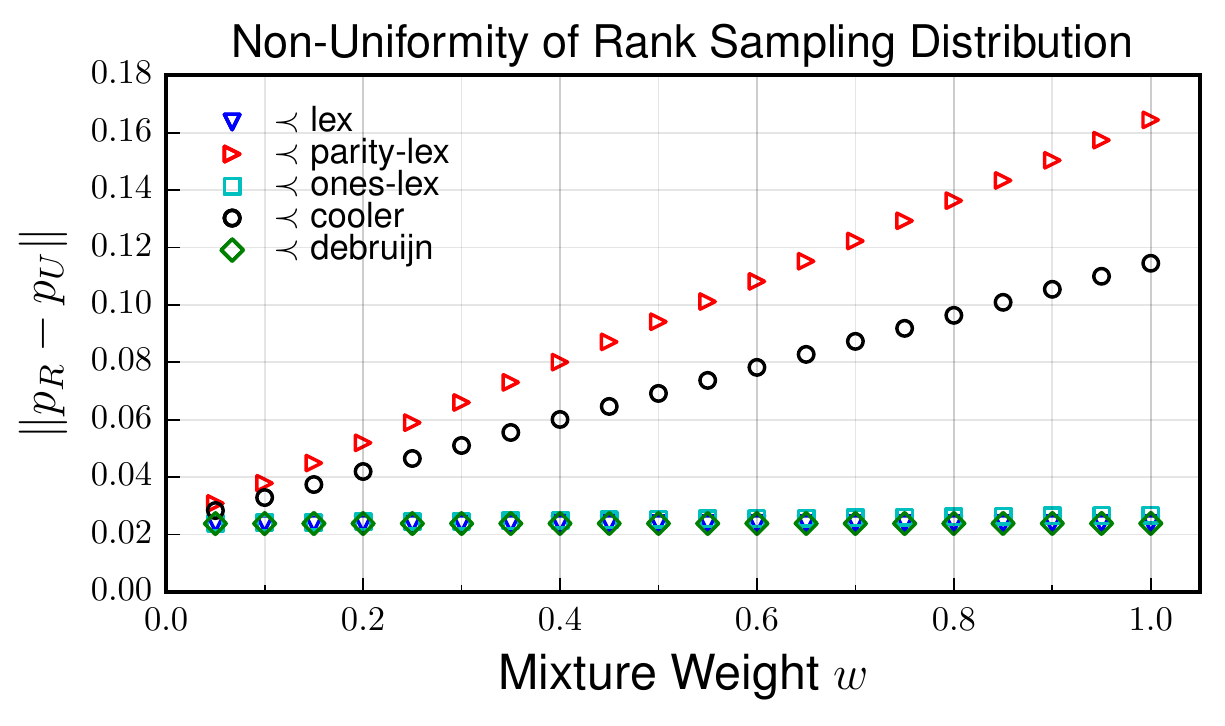}%
\end{subfigure}%
\begin{subfigure}[b]{.675\linewidth}
\includegraphics[width=\linewidth]{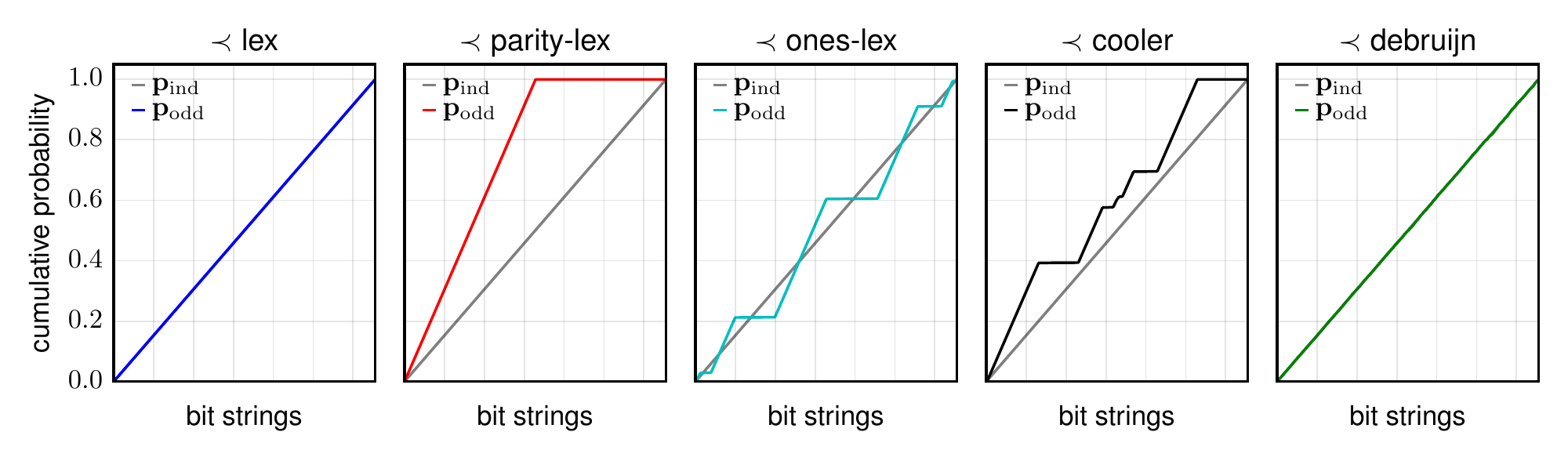}\hfill%
\end{subfigure}
\captionsetup{skip=0pt, belowskip=5pt, font=footnotesize}
\subcaption{$\bp \defas \bp_{\rm ind}\quad\quad \bq \defas w \bp_{\rm odd} + (1-w)\bp_{\rm ind}$}
\label{subfig:gof-bit-strings-odd}
\end{subfigure}

\captionsetup{skip=2pt}
\caption{In each of the two panels
\subref{subfig:gof-bit-strings-tied} and
\subref{subfig:gof-bit-strings-odd}, the left plot shows the sup-norm distance
between the sampling distribution of the rank statistic and the discrete uniform
(using Eq.~\ref{eq:dist-r0-combinatorial} in
Theorem~\ref{thm:dist-r0-combinatorial}), for a uniform null $\bp \defas
\bp_{\rm ind}$ on $\set{0,1}^{16}$ against alternative distributions of the form
$\bq \defas w\bp_{\rm alt} + (1-w)\bp_{\rm ind}$, for increasing mixture
weight $0 \le w \le 1$ and six different orderings on the binary strings. The
right plot compares the cumulative distribution function of the null
distribution (diagonal line in gray) with the cumulative distribution functions
of the alternative distribution (when $w=1$) as obtained by sorting the binary
strings according to each ordering. Orderings which induce a greater distance
between the cumulative distribution functions of the null and alternative
distributions result in more power to detect the alternative.}
\label{fig:gof-bit-strings}
\end{figure*}


\begin{figure*}[ht]
\captionsetup[subfigure]{skip=0pt}
\begin{subfigure}[b]{.675\linewidth}
\includegraphics[width=\linewidth]{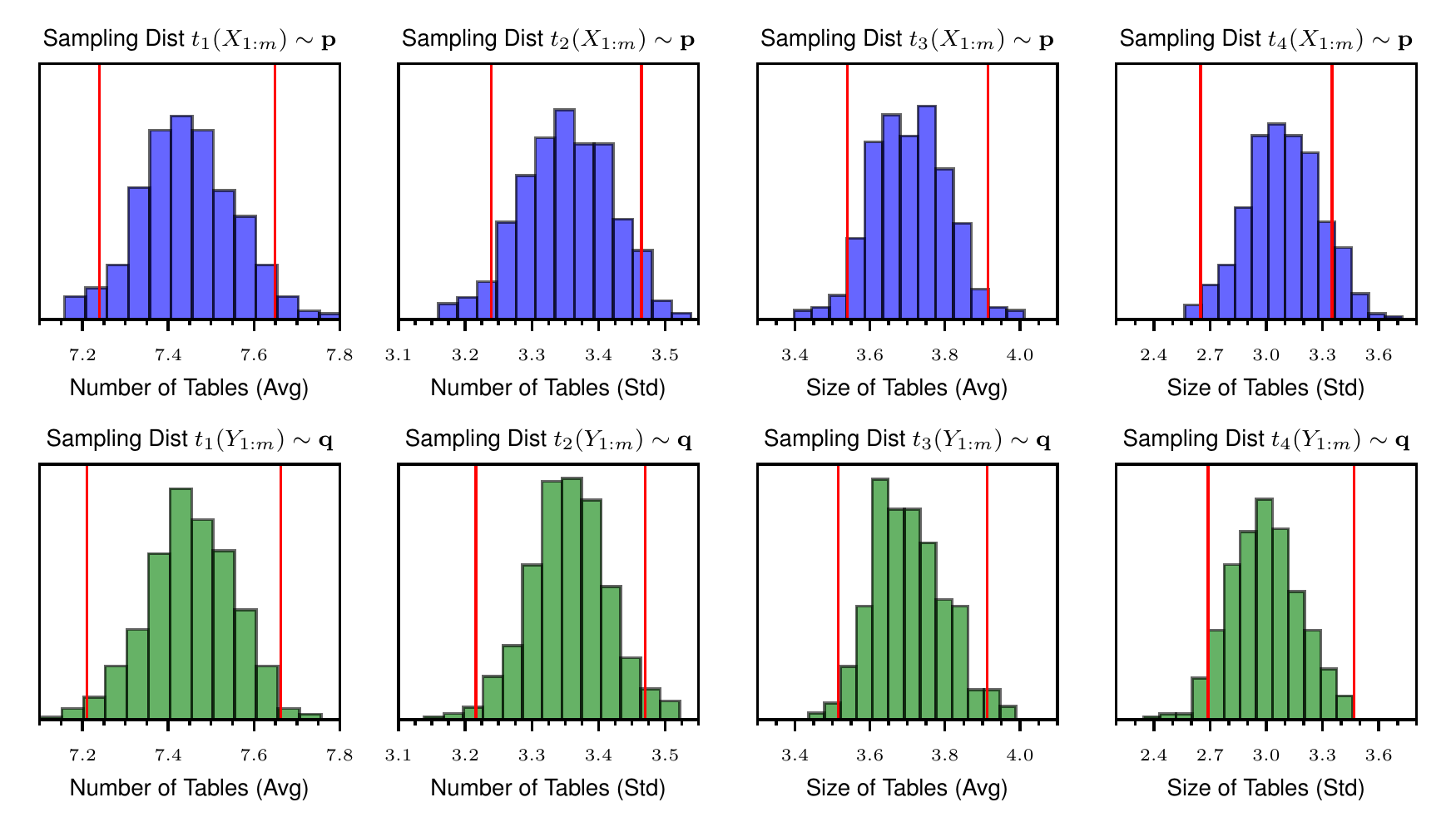}
\caption{Sampling distribution of four different probe statistics $\set{t_1,
t_2, t_3, t_4}$ of a dataset of partitions, as sampled from $\bp$ (Eq.~\eqref{eq:crp-dist-p}; blue) and from $\bq$ (Eq.~\eqref{eq:crp-dist-q}; green) estimated by
Monte Carlo simulation. Vertical red lines indicate 2.5\% and 97.5\% quantiles. Even
though $\bp\,{\ne}\,\bq$, the distributions of these statistics are aligned in such
a way that a statistic $t_j(Y_{1:m})\sim\bq$ is unlikely to appear as an extreme
value in the sampling distribution of the corresponding statistic
$t_j(X_{1:m})\sim\bp$, which leads to under-powered resampling-based tests.}
\label{subfig:pitman-yor-probe}
\end{subfigure}\hfill%
\begin{subfigure}[b]{.27\linewidth}
\includegraphics[width=\linewidth]{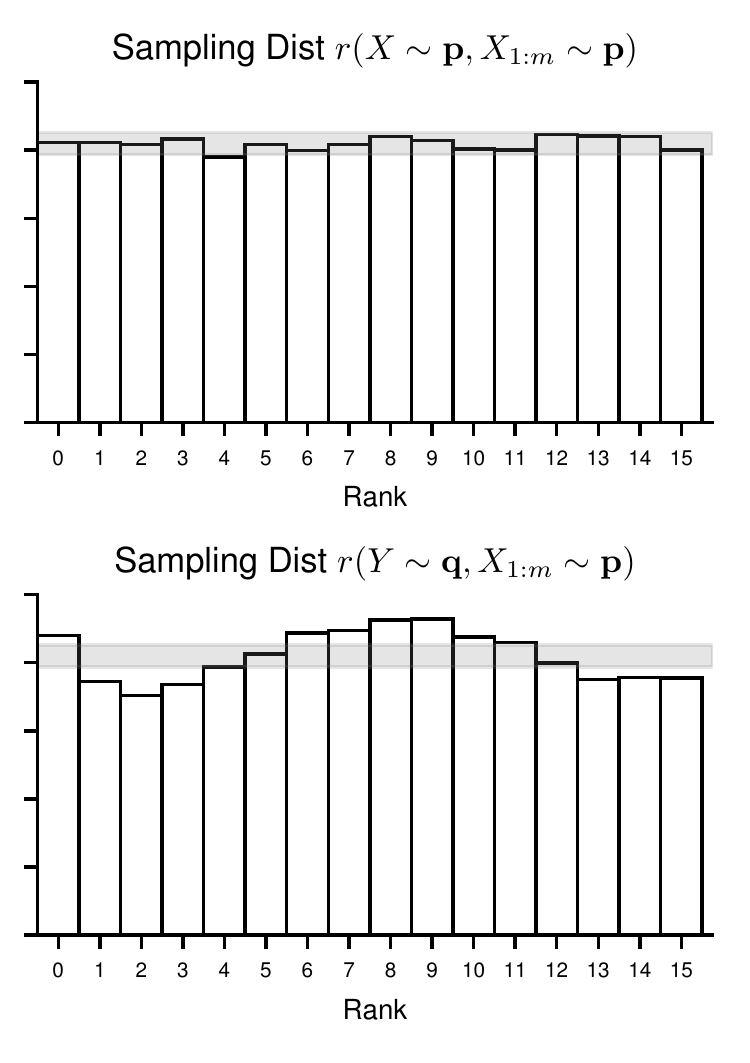}
\caption{Monte Carlo simulation of the rank statistic illustrates its
significant uniform distribution under the null hypothesis (top) and
significant non-uniform distribution under the alternative hypothesis (bottom).}
\label{subfig:pitman-yor-rank}
\end{subfigure}

\caption{Comparison of the sampling distribution of \subref{subfig:pitman-yor-probe}
various bootstrapped probe statistics~\citep{good2004} with
\subref{subfig:pitman-yor-rank} the stochastic rank statistic, for
goodness-of-fit testing the Chinese restaurant processes on $N\,{=}\,20$ customers.
Discussion in main text.}
\label{fig:pitman-yor}
\end{figure*}

We now apply the proposed test to a countable domain and two
high-dimensional finite domains, illustrating a power comparison and how
distributional differences can be detected when the number of observations is
much smaller than the domain size. We use Pearson chi-square to assess
uniformity of the SRS for Algorithm~\ref{alg:gof-uniform-rank}
(see \citet{steele2006} for alternative ways to test for a uniform null).

\subsection{Bimodal, Symmetric Poisson}
\label{subsec:example-poisson}

We first investigate the performance of the SRS for testing a pair of
symmetric, multi-modal distributions over the integers with location shift.
In particular, for $x\,{\in}\,\mathbb{Z}$, define distribution
$
\mathbf{f}(x; \lambda_1, \lambda_2) \defas \frac{1}{2}\left(
    \frac{1}{2} \mathrm{Poisson}(|x|; \lambda_1)
        + \frac{1}{2} \mathrm{Poisson}(|x|; \lambda_2)
    \right)
$.
Note $\mathbf{f}$ is a mixture of Poisson distributions with rates
$\lambda_1$ and $\lambda_2$, reflected symmetrically about $x\,{=}\,0$. We set
$\bp(x) \defas \mathbf{f}(x; 10,20)$ and $\bq \defas \mathbf{f}(x; 10,
25)$ so that $\bq$ is location-shifted in two of the four modes
(Figure~\ref{fig:poisson}, left panel).

The right plot of Figure~\ref{fig:poisson} compares the power for various sample
sizes $n$ from $\bq$ according to the SRS ($m\,{=}\,1,2,3,30$, shown in
increasing shades of gray) and several baselines (shown in color).
The baselines (AD, MMD, KS, and Mann--Whitney U) are used
to assess goodness-of-fit by performing a two-sample test on $n$ samples
from $\bq$ with samples drawn i.i.d.\ from $\bp$.
The power (at level $\alpha\,{=}\,0.05$) is estimated as the fraction of correct
answers over 1024 independent trials.
The Mann--Whitney U, which is also based on rank statistics with
a correction for ties, has no power for all $n$ as it can only detect median
shift, as does the SRS with $m\,{=}\,1$ (see
Corollary~\ref{corr:total-order-no-symmetry}).
The SRS becomes non-uniform for $m\,{=}\,2$ although this choice results in low
power. The SRS with $m\,{=}\,3$ has comparable power to the AD and MMD tests.
The SRS with $m\,{=}\,30$ is the most powerful, although it requires more
computational effort and samples from $\bp$
(Algorithm~\ref{alg:gof-uniform-rank} scales as $O(mn)$).

\subsection{Binary strings}
\label{subsec:example-binary-strings}

Let $\dom \defas \set{0,1}^k$ be the set of all length $k$ binary strings.
Define the following distributions to be uniform over all strings $x =
(x_1,\dots,x_k) \in \set{0,1}^k$ which satisfy the given predicates:
\begin{align*}
\bp_{\rm ind} &:  ~\mbox{uniform on all strings}, \\
\bp_{\rm odd} &:  ~\textstyle\sum_{i=1}^{k}x_i \equiv 1\, (\mathrm{mod}\, 2),\\
\bp_{\rm tie} &:  ~x_1 = x_2 = \cdots = x_{k/2}.
\end{align*}
Each of these distributions assigns marginal probability $1/2$ to each bit $x_i$
(for $1\le i \le k$), so all deviations from the uniform distribution $\bp_{\rm
ind}$ are captured by higher-order relationships. The five orderings used for
comparing binary strings are
\begin{align*}
\prec_{\rm lex}&: ~\mbox{Lexicographic (dictionary) ordering}, \\
\prec_{\rm par}&: ~\mbox{Parity of ones, ties broken using} \prec_{\rm lex}, \\
\prec_{\rm one}&: ~\mbox{Number of ones, ties broken using } \prec_{\rm lex}, \\
\prec_{\rm coo}&: ~\mbox{Cooler ordering (randomly generated) \citep{stevens2012coolest}}, \\
\prec_{\rm dbj}&: ~\mbox{De Bruijn sequence ordering}.
\end{align*}
We set the null distribution $\bp \defas\bp_{\rm ind}$ and construct
alternative distributions $\bq \defas w\bp_{\rm c} +  (1-w)\bp_{\rm
ind}$ as mixtures of $\bp_{\rm ind}$ with the other two distributions, where
$w \in [0, 1]$ and $c \in \set{\mathrm{odd}, \mathrm{not}}$. We take bit
strings of length $k = 16$ with $n = 256$ observations so that $\card{\dom} =
65,536$ and $0.4\%$ of the domain size is observed.

Figure~\ref{fig:gof-bit-strings} shows how the non-uniformity of the SRS
(computed using Theorem~\ref{thm:dist-r0-combinatorial}) varies
for each of the two alternatives and five orderings ($m\,{=}\,6$). Each ordering
induces a different CDF over $\set{0,1}^k$ for the alternative distribution,
shown in the right panel for $w\,{=}\,1$. Orderings with a greater maximum
vertical distance between the null and alternative CDF attain greater rank
non-uniformity. No single ordering is more powerful than all others in both test
cases. However, in each case, some ordering detects the difference even at low
weights $w$, despite the sparse observation set.

The alternative $\bq = \bp_{\rm odd}$ in Figure~\ref{subfig:gof-bit-strings-odd}
is especially challenging: in a sample, all substrings (not necessarily
contiguous) of a given length $j < k$ are equally likely.
Even though the SRS is non-uniform for all orderings, the powers vary
significantly. For example, comparing strings using $\prec_{\rm lex}$ does not
effectively distinguish between $\bp_{\rm ind}$ and $\bp_{\rm odd}$, as strings
with an odd number of ones are lexicographically evenly interspersed within the
set of all strings. The parity ordering (which is optimal for this alternative)
and the randomly generated cooler ordering have increasing power as $w$
increases.

\subsection{Partition testing}
\label{subsec:example-partition-testing}

We next apply the SRS to test distributions on the space of partitions of the
set $\set{1,2,\dots,N}$. Let $\Pi_N$ denote the set of all such partitions. We
define a distribution on $\Pi_N$ using the two-parameter Chinese Restaurant
Process (CRP) \citep[Section~5.1]{buntine2010}. Letting
$(x|y)_N \defas (x)(x+y)\cdots(x+(N-1)y)$, the probability of a partition
$\pi \defas \set{\pi_1,\dots,\pi_k} \in \Pi_N$ with $k$ tables (blocks) is
given by
\begin{align*}
\pitmanyor(\pi ; a, b) \defas \begin{cases}
  \frac{(b{|}a)_k}{(b{|}1)_N} \prod_{i=1}^{k}(1-a)_{c_k-1} & (\mbox{if }a > 0)\\[10pt]
  \frac{b^k}{(b{|}1)_N} \prod_{i=1}^{k}(c_k-1)! & (\mbox{if }a = 0),
  \end{cases}
\end{align*}
where $c_i$ is the number of customers (integers) at table $\pi_i$
($1\,{\le}i\,{\le}\,k$).
Simulating a CRP proceeds by sequentially
assigning customers to tables \citep[Def.~7]{buntine2010}.
Even though we can compute the probability of any partition, the cardinality of
$\Pi_N$ grows exponentially in $N$ (e.g., $\abs{\Pi_{20}} \approx
5.17\,{\times}\,{10}^{13}$). The expected frequency of any partition is
essentially zero for sample size $n\,{\ll}\,\abs{\Pi_{N}}$, so Pearson
chi-square or likelihood-ratio tests on the raw data are inappropriate.
Algorithm~\ref{alg:partition-ordering} defines a total order on the
partition domain $\Pi_N$.


\begin{algorithm}[ht]
\scriptsize
\captionsetup{font=footnotesize}
\caption{Total order $\prec$ on the set of partitions $\Pi_N$}
\label{alg:partition-ordering}
\begin{algorithmic}[1]
\algrenewcommand\algorithmicindent{1.0em}%
\Require{$\left\{ \begin{tabular}{l}
  Partition $\pi \defas \set{\pi_1,\pi_2,\dots,\pi_k} \in \Pi_N$
    with $k$ blocks.\\
  Partition $\nu \defas \set{\nu_1,\nu_2,\dots,\nu_l} \in \Pi_N$
    with $l$ blocks.
  \end{tabular}
  \right.$}
\Ensure{
  $\LT$ if $\pi \prec \nu$;
  $\GT$ if $\pi \succ \nu$;
  $\EQ$ if $\pi  = \nu$.
}
\If{$k < l$}
  \Return \LT \Comment{$\nu$ has more blocks}
  \EndIf
\If{$k > l$}
  \Return \GT \Comment{$\pi$ has more blocks}
  \EndIf
\State $\tilde{\pi} \gets$ blocks of $\pi$ sorted by value of least element in the block
\State $\tilde{\nu} \gets$ blocks of $\nu$ sorted by value of least element in the block
\For{$b=1,2,\dots,l$}
  \If{$\abs{\tilde{\pi}_b} < \abs{\tilde{\nu}_b}$}
    \Return \LT \Comment{$\tilde{\nu}_b$ has more elements}
    \EndIf
  \If{$\abs{\tilde{\pi}_b} > \abs{\tilde{\nu}_b}$}
    \Return \GT \Comment{$\tilde{\pi}_b$ has more elements}
    \EndIf
  \State $\pi'_b \gets$ values in $\tilde\pi_b$ sorted in ascending order
  \State $\nu'_b \gets$ values in $\tilde\nu_b$ sorted in ascending order
  \For{$i=1,2,\dots, \abs{\pi'_b}$}
    \If{$\pi'_{b,i} < \nu'_{b,i}$}
      \Return \LT \Comment{$\pi'_b$ has smallest element}
      \EndIf
    \If{$\pi'_{b,i} > \nu'_{b,i}$}
      \Return \GT \Comment{$\nu'_b$ has smallest element}
      \EndIf
  \EndFor
\EndFor
\State \Return \EQ
\end{algorithmic}
\end{algorithm}

We consider the following pair of distributions:
\begin{align}
\bp &\defas \pitmanyor(0.26, 0.76)/2 + \pitmanyor(0.19, 5.1)/2
  \label{eq:crp-dist-p}\\
\bq &\defas \pitmanyor(0.52, 0.52). \label{eq:crp-dist-q}
\end{align}
These distributions are designed to ensure that partitions from $\bp$ and $\bq$
have similar distributions on the number and sizes of tables.
Figure~\ref{subfig:pitman-yor-probe} shows a comparison of using Monte Carlo
simulation of various bootstrapped probe statistics for assessing
goodness-of-fit versus using the SRS with the ordering in
Algorithm~\ref{alg:partition-ordering}.

In Figure~\ref{subfig:pitman-yor-probe}, each probe statistic takes a size $m$
dataset $X_{1:m}$ (where each $X_i$ is a partition) and produces a numerical
summary such as the average of the number of tables in each sample. A resampling
test \citep{good2004} that uses these probe statistics will report (with high
probability) that an observed statistic $t(Y_{1:m}) \sim \bq$ drawn from the
alternative distribution is a non-extreme value in the null distribution
$t(X_{1:m}) \sim \bp$ (as indicated by alignment of their quantiles, shown in
red) and will therefore have insufficient evidence to reject $\bp = \bq$.

On the other hand, Figure~\ref{subfig:pitman-yor-rank} shows that when ranked
using the ordering obtained from Algorithm~\ref{alg:partition-ordering} (which
is based on a multivariate combination of the univariate probe statistics in
Figure~\ref{subfig:pitman-yor-probe} specified procedurally), a partition $Y
\sim \bq$ is more likely to lie in the center of a dataset $X_{1:m}
\simiid \bp$, as illustrated by the non-uniform rank distribution under the
alternative hypothesis (the gray band shows 99\% variation for a uniform
histogram). By comparing the top and bottom panels of
Figure~\ref{subfig:pitman-yor-rank}, the SRS shows that partitions from $\bq$
have a poor fit with respect to partitions from $\bp$, despite their agreement
on multiple univariate summary statistics shown in
Figure~\ref{subfig:pitman-yor-probe}.


\section{Applications}
\label{sec:applications}

We next apply the proposed test to assess the sample
quality of random data structures obtained from approximate sampling algorithms over
combinatorially large domains with intractable probabilities.

\subsection{Dirichlet process mixture models}
\label{subsec:applications-dpmm}

\begin{figure}[!htb]
\includegraphics[width=\linewidth]{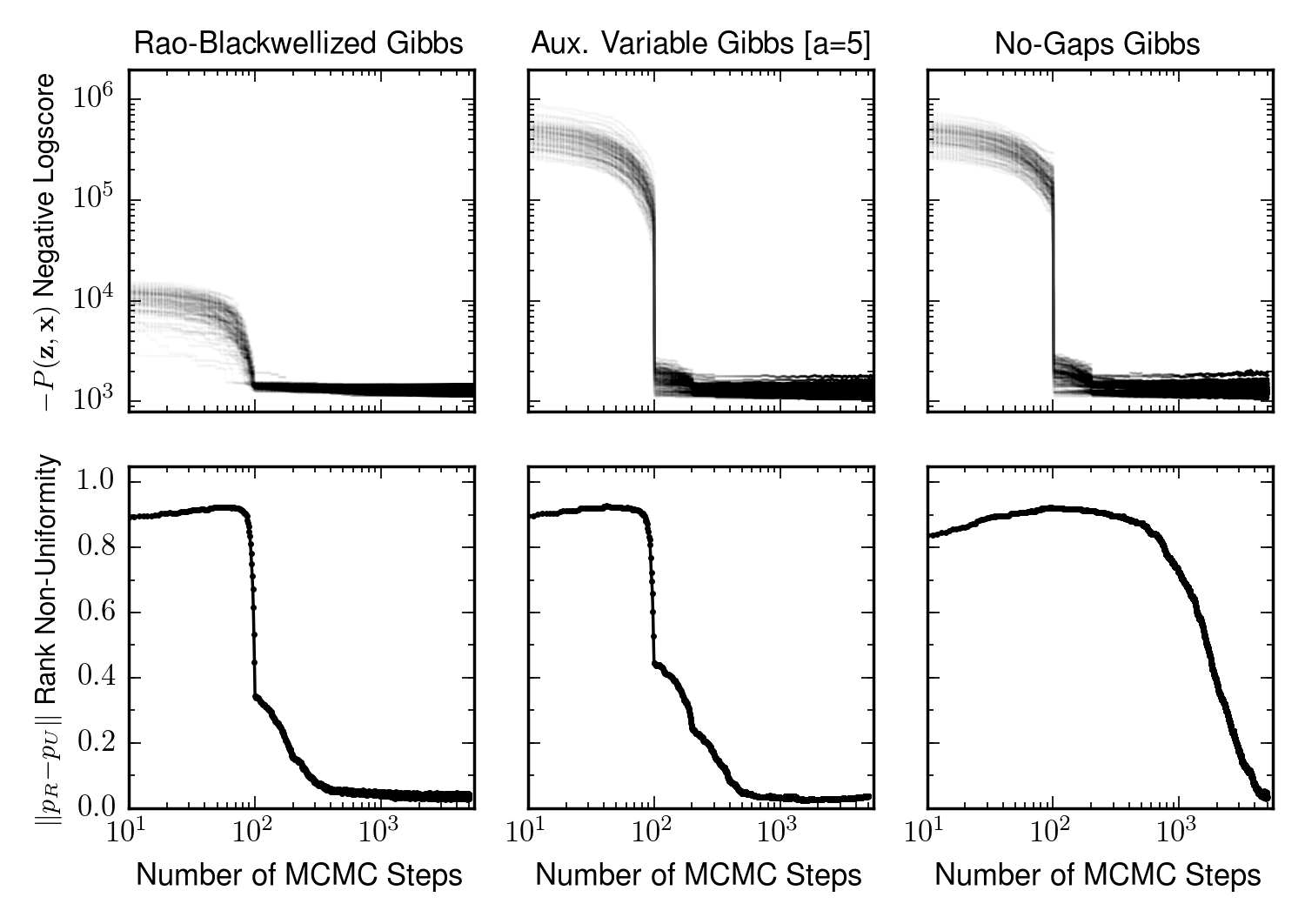}
\caption{The uniformity of the SRS (bottom row) captures convergence behavior of
MCMC sampling algorithms for Dirichlet process mixture models that are not
captured by standard diagnostics such as the logscore (top row).}
\label{fig:dpmm-gibbs}
\end{figure}


\begin{figure*}[ht]
\captionsetup[subfigure]{skip=0pt}
\begin{subfigure}{.5\linewidth}
\includegraphics[width=\linewidth]{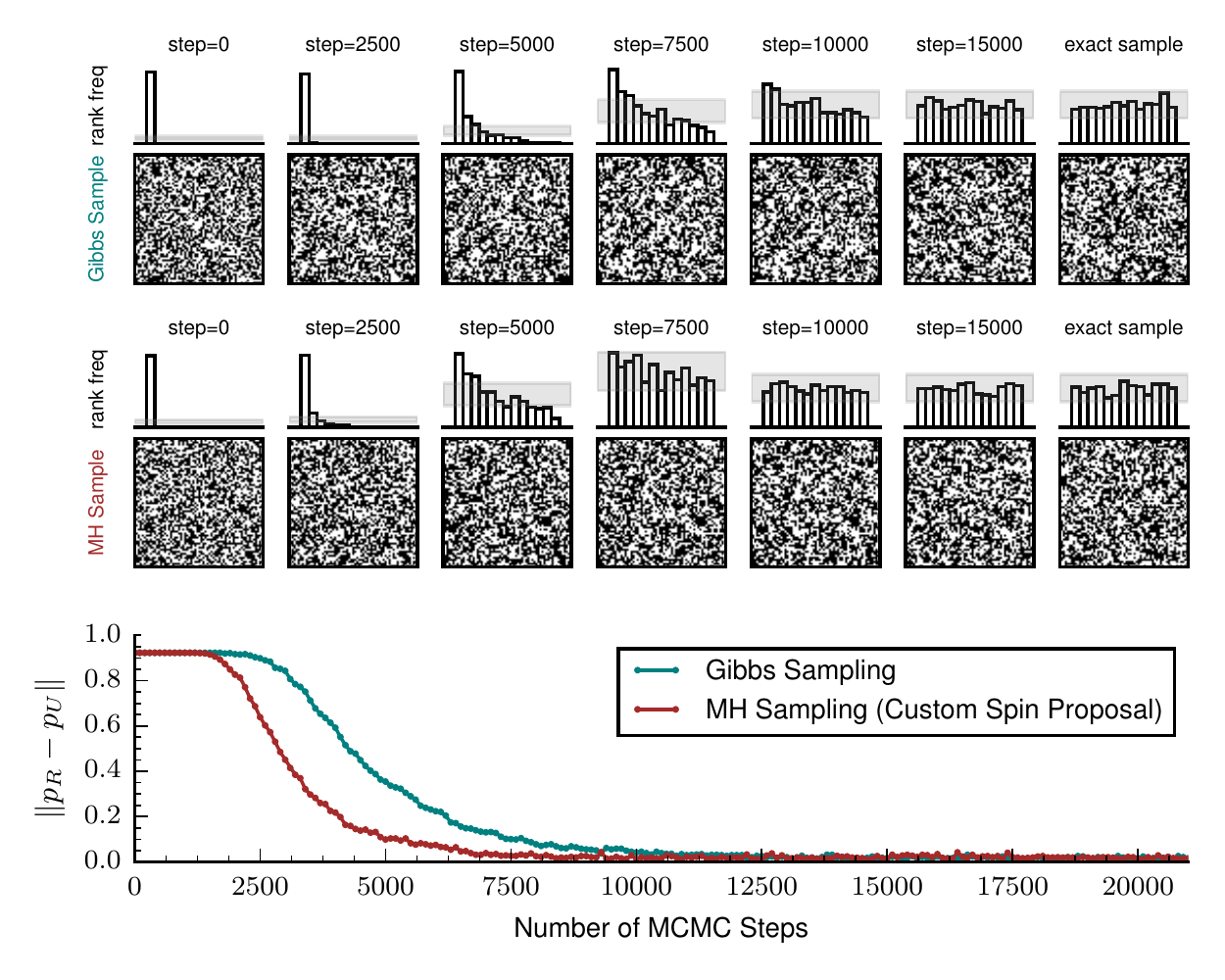}
\subcaption{Temperature $T = 8$}
\label{subfig:ising-t8}
\end{subfigure}
\begin{subfigure}{.5\linewidth}
\includegraphics[width=\linewidth]{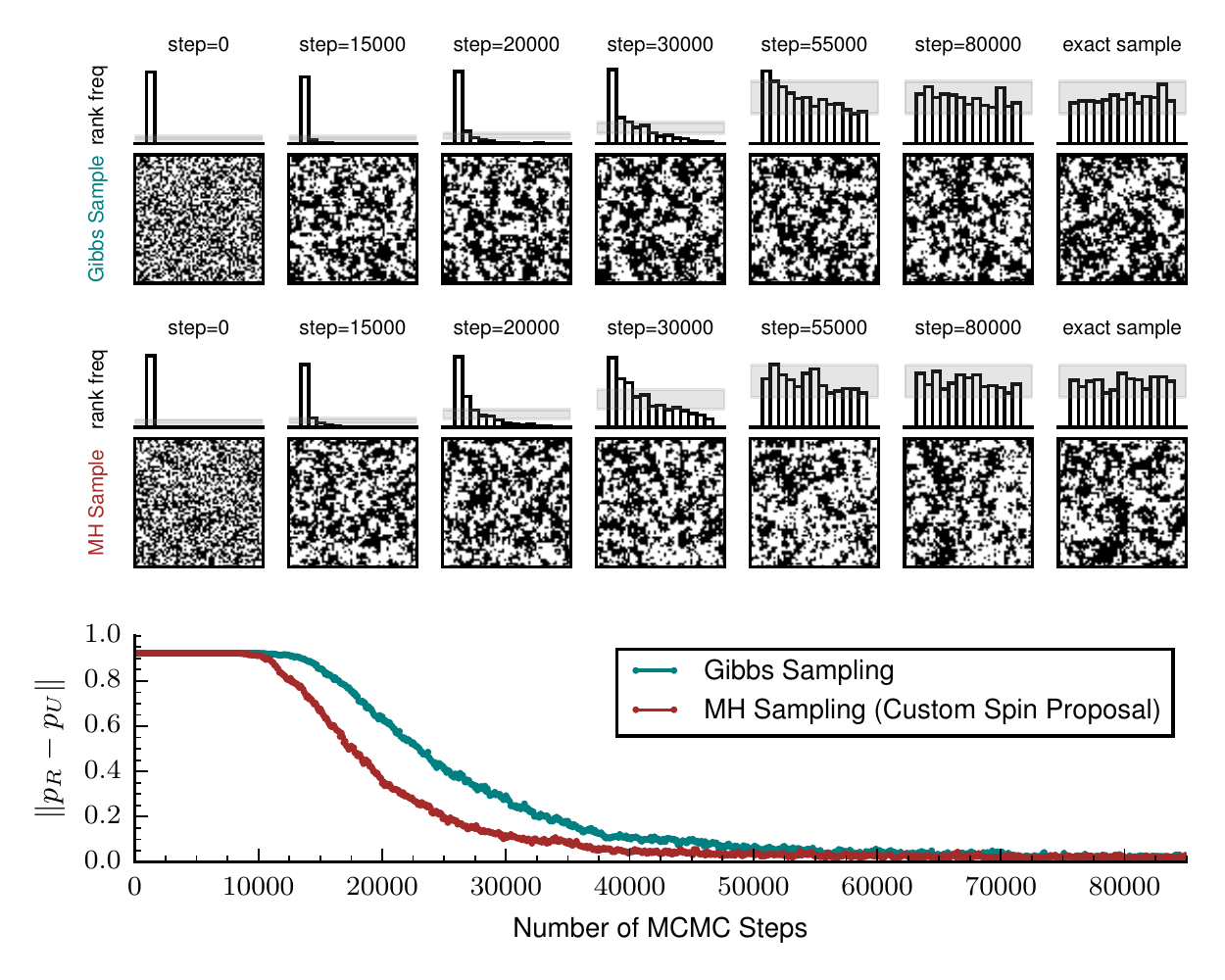}
\subcaption{Temperature $T = 3$}
\label{subfig:ising-t3}
\end{subfigure}%
\caption{Assessing the goodness-of-fit of approximate samples of a $64 \times
64$ Ising model for Gibbs sampling and Metropolis--Hastings sampling
(with the custom spin proposal from \citep{mackay2002}) at two
temperatures using the SRS. In both cases, the SRS converges to its uniform
distribution more rapidly for samples obtained from MH than for those from Gibbs sampling.}
\label{fig:ising}
\end{figure*}

The recent paper \citet{talts2018} describes simulation-based calibration (SBC),
a procedure for validating samples from algorithms that can generate posterior
samples for a hierarchical Bayesian model. More specifically, for a prior
$\pi(z)$ over the parameters $z$ and likelihood function $\pi(x|z)$ over data
$x$, integrating the posterior over the joint distribution returns the prior
distribution:
\begin{align}
\pi(z) = \int
  \left[\pi(z | x')\pi(x'|z')\mathrm{d}x'\right]
  \pi(z')\mathrm{d}z'.
  \label{eq:data-avg-posterior}
\end{align}
Eq.~\eqref{eq:data-avg-posterior} indicates that by simulating $n$ datasets
$\set{x_1,\dots,x_n}$ i.i.d.\ from the marginal distribution, samples
$\set{\hat{z}_1,\dots,\hat{z}_n}$ (where $z_i\,{\approx}\,\pi(z|x_i)$) from an
approximate posterior should be i.i.d.\ samples from the prior $\pi(z)$.
An approximate sampler can be thus be diagnosed by performing a goodness-of-fit
test to check whether $\hat{z}_{1:n}$ are distributed according to $\pi$.
Ranks of univariate marginals of a continuous parameter
vector $z\,{\in}\,\mathbb{R}^d$
are used in \citet{talts2018}. We extend SBC to handle discrete latent
variables $z$ taking values in a large domain.

We sampled $n\,{=}\,1000$ datasets $\set{x_1,\dots,x_n}$
independently from a Dirichlet process mixture model. Each dataset $x_i$ has
$k\,{=}\,100$ observations and each observation is five-dimensional
(i.e., $x_i\,{\in}\,\mathbb{R}^{k\times5}$) with a Gaussian likelihood.
From SBC, samples $\hat{z}_{1:n}$ (where $z_i\,{\in}\,\Pi_{k}$ and
$|\Pi_{k}|\,{\approx}\,10^{115})$ of the mixture assignment vector should be
distributed according to the CRP prior $\pi(z)$.
The top row of Figure~\ref{fig:dpmm-gibbs} shows trace plots of the logscore
(unnormalized posterior) of approximate samples from Rao--Blackwellized Gibbs,
Auxiliary Variable Gibbs, and No-Gaps Gibbs samplers (Algorithms 3, 8, and 4 in
\citet{neal2000}). Each line corresponds to an independent run of MCMC
inference. The bottom row shows the evolution of the uniformity of the SRS using
$m\,{=}\,64$ and the ordering on partitions from
Algorithm~\ref{alg:partition-ordering}.

While logscores typically stabilize after 100 MCMC steps (one epoch through all
observations in a dataset) and suggest little difference across the three
samplers, the SRS shows that Rao-Blackwellized Gibbs is slightly more efficient
than Auxiliary Variable Gibbs and that the sample quality from No-Gaps Gibbs is
inferior to those from the other two algorithms up until roughly $5,000$ steps.
These results are consistent with the observation from \citet{neal2000} that
No-Gaps has inefficient mixing (it excessively rejects proposals on singleton
clusters).

\subsection{Ising models}
\label{subsec:applications-ising}

In this application we use the SRS to assess the sample quality of approximate
Ising model simulations. For a ferromagnetic $k \times k$ lattice with
temperature $T$, the probability of a spin configuration $x \in
\set{-1, +1}^{k \times k}$ is
\begin{align}
P(x) \propto \exp{\left(-1/T\textstyle\sum_{i,j}x_ix_j\right)}.
\label{eq:ising}
\end{align}
While Eq.~\eqref{eq:ising} is intractable to compute for any $x$ due to the
unknown normalization constant, coupling-from-the-past \citep{propp1996} is a
popular MCMC technique which can tractably obtain exact samples from the Ising
model. For a $64 \times 64$ Ising model (domain size $2^{64\times{64}}$), we
obtained 650 exact samples using coupling-from-the-past, and used these
``ground-truth'' samples to assess the goodness-of-fit of approximate samples
obtained via Gibbs sampling and Metropolis--Hastings sampling (with a custom
spin proposal \citep[Section~31.1]{mackay2002}).

For each temperature $T\,{=}\,3$ and $T\,{=}\,8$, we obtained 7,800 approximate
samples using MH and Gibbs. The first two rows of Figure~\ref{fig:ising} each
show the evolution of one particular sample (Gibbs, top; MH, bottom). Two exact
samples are shown in the final column of each panel. All approximate and exact
samples are independent of one another, obtained by running parallel Markov
chains. The SRS of the exact samples with respect to the approximate samples was
taken at checkpoints of 100 MCMC steps, using $m\,{=}\,12$ and an ordering based
on the Hamiltonian energy, spin magnetization, and connected components. SRS
histograms (and $99\%$ variation bands) evolving at various steps are shown
above the Ising model renderings.

The SRS is non-uniform (including in regimes where the difference between
approximate and exact samples is too fine-grained to be detected visually) at
early steps and more uniform at higher steps. The plots show that MH is a more
efficient sampler than Gibbs at moderate temperatures, as its sample quality
improves more rapidly. This characteristic was conjectured in
\citet{mackay2002}, which noted that the MH sampler ``has roughly double the the
probability of accepting energetically unfavourable moves, so may be a more
efficient sampler [than Gibbs]''. In addition, the plots suggest that the
samples become close to exact (in terms of their joint energy, magnetization,
and connected components characteristics) after 20,000 steps for $T\,{=}\,8$ and
100,000 steps for $T\,{=}\,3$, even though obtaining exact samples using
coupling-from-the-past requires between 500,000 and 1,000,000 MCMC steps for
both temperatures.


\section{Conclusion}
\label{sec:conclusion}

This paper has presented a flexible, simple-to-implement, and consistent
goodness-of-fit test for discrete distributions.
The test statistic is based on the ranks of observed samples with respect to new
samples from the candidate distribution.
The key insight is to compute the ranks
using an ordering on the domain that is able to detect differences in properties
of interest in high dimensions.
Unlike most existing statistics, the SRS is distribution-free and has a simple
exact sampling distribution.
Empirical studies indicate that the SRS is a valuable addition to the
practitioner's toolbox for assessing sample quality in regimes which are not
easily handled by existing methods.


\section*{Acknowledgments}
\label{sec:acks}

The authors thank the anonymous referees for their helpful feedback.
This research was supported by the DARPA SD2 program (contract
\mbox{FA8750-17-C-0239}); the Ethics and Governance of Artificial
Intelligence Initiative of the MIT Media Lab and
Harvard's Berkman Klein Center; the
Systems That Learn Initiative of MIT CSAIL; and an anonymous philanthropic gift.

\bibliography{main.bib}

\clearpage
\appendix


\section{Appendix: Proofs}
\label{sec:proofs}

\subsection{Uniformity of rank}
\label{subsec:proposed-statistic-uniform}

Throughout this appendix,
let $\dom$ be a non-empty finite or countably infinite set,
let $\prec$ be a total order on $\dom$ (of any order type),
and let $\bp$ and $\bq$ each be a probability distribution on $\dom$.
For $n\in\Naturals$, let $[n]$ denote the set $\set{0,1,2,\dots,n-1}$.

Given a positive integer $m$, define the following random variables:
\begin{align}
X_0 &\sim \bq
  \label{eq:proc-sim-xq-appx} \\
U_0 &\sim \mathsf{Uniform}(0,1)
  \label{eq:proc-sim-uq-appx} \\
X_1, X_2, \dots, X_m &\simiid \bp
  \label{eq:proc-sim-xp-appx} \\
U_1, U_2, \dots, U_m &\simiid \mathsf{Uniform}(0,1)
  \label{eq:proc-sim-up-appx} \\
R = \textstyle \sum_{j=1}^m \indicator{X_j \prec X_0}
  &+ \indicator{X_j = X_0, U_j < U_0}.
  \label{eq:proc-stochastic-rank-appx}
\end{align}

Our first main result is the following, which establishes necessary and
sufficient conditions for uniformity of the rank statistic.

\vspace*{5pt}
\begin{theorem}[Theorem~\ref{thm:uniformity} in the main text]
\label{thm:uniformity-appx}
We have $\bp\,{=}\,\bq$ if and only if for all $m\,{\ge}\,1$, the rank statistic
$R$ is uniformly distributed on $\mbox{$[m+1]\,{\defas}\,\set{0,1,\dots,m}$}$.
\end{theorem}

Before proving Theorem~\ref{thm:uniformity-appx}, we state and prove several
lemmas. We begin by showing that an i.i.d.\ sequence yields a uniform rank
distribution.

\vspace*{5pt}
\begin{lemma}
\label{lem:uniformity-rank-statistics}
Let $T_0, T_1, \dots, T_m$ be an i.i.d.\ sequence of random variables.
If $\Prob{T_i = T_j} = 0$ for all distinct $i$ and $j$, then the rank
statistics $S_i \defas \sum_{j=0}^{m} \mathbb{I}[T_j \plt T_i]$
for $0 \le i \le m$ are each uniformly distributed on $[m+1]$.
\end{lemma}

\begin{proof}
Since $T_0, T_1, \ldots, T_m$ is i.i.d., it is a finitely exchangeable sequence,
and so the rank statistics $S_0, \dots, S_m$ are identically (but not
independently) distributed.

Fix an arbitrary $k \in [m+1]$. Then $\Prob{S_i = k} = \Prob{S_j = k}$ for all
$i, j \in [m+1]$. By hypothesis, $\Prob{T_i = T_j} = 0$ for distinct $i$ and
$j$. Therefore the rank statistics are almost surely distinct, and the events
$\set{S_i = j}$ (for $0\le i \le m$) are mutually exclusive and exhaustive.
Since these events partition the outcome space, their probabilities sum to $1$,
and so $\Prob{S_i=k} = 1/(m+1)$ for all $i \in [m+1]$.

Because $k$ was arbitrary, $S_i$ is uniformly distributed on $[m+1]$ for all
$i\in [m+1]$.
\end{proof}

We will also use the following result about convergence of discrete uniform
variables to a continuous uniform random variable.

\vspace*{5pt}
\begin{lemma}
\label{lem:discrete-uniform-continuous}
Let $(V_m)_{m\ge 1}$ be a sequence of discrete random variables such that $V_m$
is uniformly distributed on $\set{0,1/m,2/m,\dots,1}$, and let $U$ be a
continuous random variable uniformly distributed on the interval $[0,1]$. Then
$(V_m)_{m\ge 1}$ converges in distribution to $U$, i.e.,
\begin{align}
\lim_{m \to \infty} \Prob{V_m < u} = \Prob{U < u} = u.
\label{eq:discrete-uniform-convergence}
\end{align}
for all $u \in [0,1]$.

Furthermore, the convergence \eqref{eq:discrete-uniform-convergence} is uniform
in $u$.
\end{lemma}

\begin{proof}
Let $\epsilon > 0$.
The distribution function $F_m$ of $V_m$ is given by
\begin{align*}
F_m(u) = \begin{cases}
  1/(m+1) & u \in [0, 1/m) \\
  2/(m+1) & u \in [1/m, 2/m) \\
  \cdots \\
  (a+1)/(m+1) & u \in [a/m, (a+1)/m) \\
  \cdots \\
  m/(m+1) & u \in [(m-1)/m, 1) \\
  1 & u = 1.
\end{cases}
\end{align*}

Observe that for $0 \le a < m$, the value $F_m(u)$ lies in the interval $[a/m,
(a+1)/m)$ since we have that
\mbox{$(a/m) < (a+1)/(m+1) < (a+1)/m$}. Since $u$ is also in this
interval, $\abs{F_m(u) - u} \le (a+1)/m - a/m = 1/m < \epsilon$ whenever $m >
1/\epsilon$, for all $u$.
\end{proof}

The following intermediate value lemma for step functions on the rationals is
straightforward. It makes use of sums defined over subsets of the rationals,
which are well-defined, as we discuss in the next remark.

\vspace*{5pt}
\begin{lemma}
\label{lem:intermediate-value-rats}
Let $p \colon (\Rationals \cap [0,1]) \to [0,1]$ be a function
satisfying $p(0) = 0$ and $\sum_{x \in \Rationals \cap [0,1]}p(x) = 1$.
Then for each $\delta \in (0,1)$, there is some $w \in \Rationals \cap [0,1]$
such that
\begin{align*}
\sum_{x \in \Rationals \cap (0,w)}p(x) \le \delta \le \sum_{x \in \Rationals \cap (0, w]}p(x).
\end{align*}
\end{lemma}

\vspace*{5pt}
\begin{remark}
\label{rem:series-indexed-rats}
The infinite sums in Lemma~\ref{lem:intermediate-value-rats} taken over a subset
of the rationals can be formally defined as follows: Consider an arbitrary
enumeration $\set{q_1, q_2, \dots, q_n, \dots}$ of $\mathbb{Q} \cap [0,1]$, and
define the summation over the integer-valued index $n \ge 1$. Since the series
consists of positive terms, it converges absolutely, and so all rearrangements
of the enumeration converge to the same sum (see, e.g.,
\citet[Theorem~3.55]{rudin1976principles}).

One can show that the Cauchy criterion holds in this setting. Namely, suppose
that a sum $\sum_{a < x < c} p(x)$ of non-negative terms converges. Then for all
$\epsilon > 0$ there is some rational $b \in (a, c)$ such that
$\sum_{a < x \le b} p(x) < \epsilon$.
\end{remark}

We now prove both directions of Theorem~\ref{thm:uniformity-appx}.

\begin{proof}[Proof of Theorem~\ref{thm:uniformity-appx}]
Because $\dom$ is countable, by a standard back-and-forth argument the total
order \mbox{$(\dom, \prec)$} is isomorphic to $(B, <)$ for some subset $B
\subseteq \Rationals \cap (0,1)$. Without loss of generality, we may therefore
take $\dom$ to be $\Rationals \cap [0,1]$ and assume that $\bp(0) = \bp(1) = 0$.

Consider the unit square $[0,1]^2$ equipped with the dictionary order
$\dictorder$. This is a total order with the least upper bound property. For
each $i\in [m+1]$, define $T_i \defas (X_i, U_i)$, which takes values in
$[0,1]^2$, and observe that the rank $R$ in Eq.~\eqref{eq:proc-stochastic-rank}
of Theorem~\ref{thm:uniformity-appx} is equivalent to the rank
$\sum_{i=0}^{m}\mathbb{I}[T_i \dictorder T_0]$ of $T_0$ taken according to the
dictionary order.

\paragraph{\bfseries (Necessity)} Suppose $\bp=\bq$. Then $T_0, \dots, T_m$ are
independent and identically distributed. Since $U_0, \dots, U_m$ are continuous
random variables, we have $\Prob{T_i = T_j} =0$  for all $i
\ne j$. Apply Lemma~\ref{lem:uniformity-rank-statistics}.

\paragraph{\bfseries (Sufficiency)} Suppose that for all $m > 0$, we have
that the rank $R$ is uniformly distributed on $\set{0,1,2,\dots,m}$. We begin
the proof by first constructing a distribution function $F_\bp$ on the unit
square and then establishing several of its properties. First let
$\tbp\colon [0,1] \to [0,1]$ be the ``left-closed right-open''
cumulative distribution function of $\bp$, defined by
\begin{align*}
\tilde{\bp}(x) &\defas \sum_{y \in \Rationals \cap [0,x)}\bp(y)
\end{align*}
for $x \in [0,1]$.
Define $\bp'$ to be the probability measure on $[0,1]$ that is equal to $\bp$ on
subsets of $\Rationals \cap [0,1]$ and is null elsewhere, and define the
distribution function $F_{\bp}\colon [0,1]^2 \to [0,1]$ on $S$ by
\begin{align*}
F_\bp(x, u) \defas \tbp(x) + u\bp'(x)
\end{align*}
for $(x,u) \in [0,1]^2$. To establish that $F_\bp$ is a valid distribution
function, we show that its range is $[0,1]$; it is monotonically non-decreasing
in each of its variables; and it is right-continuous in each of its variables.

It is immediate that $F_\bp(0, 0) = 0$ and $F_\bp(1, 1) = 1$. Furthermore, To
establish that $F_\bp$ is monotonically non-decreasing, put $x < y $ and
$ u < v$ and observe that
\begin{align*}
F_\bp(x, u)
  &= \tbp(x) + u\bp'(x) \\
  &\le \tbp(x) + \bp'(x) \\
  &\le \textstyle\sum_{z \in \Rationals \cap [0, y)} \bp'(z) \\
  &= \tbp(y) \\
  &\le F_\bp(y, u)
\shortintertext{and}
F_\bp(x, u)
  &= \tbp(x) + u\bp'(x) \\
  &\le \tbp(x) + v\bp'(x) \\
  &= F_\bp(x, v).
\end{align*}

We now establish right-continuity. For fixed $x$, $F_\bp(x,u)$ is a linear
function of $u$ and so continuity is immediate. For fixed $u$, we have shown
that $F_\bp(x,u)$ is non-decreasing so it is sufficient to show that for any $x$
and for any $\epsilon > 0$ there exists $x' > x$ such that
\begin{align*}
\epsilon &> F(x', u)  - F(x,u) \\
  &= \tbp(x') + u \bp'(x') - \tbp(x) - u \bp(x) \\
  &= \tbp(x') + u \bp'(x') - \tbp(x) - u \bp(x) \\
  &= \sum_{y \in \mathbb{Q}\cap[x,x']} \bp(y),
\end{align*}
which is immediate from the Cauchy criterion.

Finally, we note that Lemma~\ref{lem:intermediate-value-rats} and the continuity
of $F_\bp$ in $u$ together imply that $F_\bp$ obtains all intermediate values,
i.e., for any $\delta \in [0,1]$ there is some $(x, u)$ such that
$F(x,u) = \delta$.

Next define the inverse $F_\bp^{-1}\colon [0,1] \to [0,1]^2$ by
\begin{align}
F_\bp^{-1}(s) \defas \inf\set{(x,u) \mid F_\bp(x,u) = s}
\label{eq:generalized-inverse-function}
\end{align}
for $s \in [0,1]$, where the infimum is taken with the respect to the dictionary
order $\dictorder$. The set in Eq~\eqref{eq:generalized-inverse-function} is
non-empty since $F_\bp$ obtains all values in $[0,1]$. Moreover, $F_\bp^{-1}(s)
\in [0,1]^2$ since $\dictorder$ has the least upper bound property. (This
``generalized'' inverse is used since $F_{\bp}$ is one-to-one only under the
stronger assumption that $\bp(x) > 0$ for all $x \in \Rationals \cap (0,1)$.)
Analogously define $F_\bq$ in terms of $\bq$.

Now define the rank function
\begin{align*}
r(a_0, \set{a_1, \dots, a_m}) \defas \sum_{i=0}^{m}\mathbb{I}[a_i < a_0]
\end{align*}
and note that $R \equiv r(T_0, \set{T_1, \dots, T_m})$.
By the hypothesis, $r(T_0, \set{T_1, \dots, T_m})/m$ is uniformly distributed on
$\set{0,1/m,2/m,\dots,1}$ for all $m > 0$.
Applying Lemma~\ref{lem:discrete-uniform-continuous} gives
\begin{align}
&\hspace*{-20pt}
\lim_{m\to\infty} \Prob{\frac{1}{m}\tilde{r}(T_0, \set{T_1,\dots,T_m}) < s} \notag \\
&\quad= \Prob{U_0 < s} \notag \\
&\quad= s. \label{eq:rank-converge-distribution-uniform}
\end{align}
for $s \in [0,1]$.

For any $t \in [0,1]$ and $m\ge 1$, the random variable $\hat{F}^m_\bp(t)
\defas \tilde{r}(t, \set{T_1,\dots,T_m})/m$ is the empirical distribution
of $F_\bp$.
Therefore, by the Glivenko--Cantelli theorem for empirical distribution
functions on $k$-dimensional Euclidean space
\citep[Corollary of Theorem~4]{dehardt1971}, the sequence of random variables
$(\hat{F}^m_\bp(t))_{m\ge 1}$ converges a.s.\ to the real number
$F_\bp(t)$ uniformly in $t$,
Hence the sequence $(\hat{F}^m_\bp(T_0))_{m\ge 1}$ converges a.s.\ to the random
variable $\hat{F}_\bp(T_0)$, so that for any $s \in [0,1]$,
\begin{align}
&
\hspace*{-30pt}
\lim_{m\to\infty} \Prob{
    \frac{1}{m}\tilde{r}(T_0, \set{T_1,\dots,T_m}) < s}
    \notag \\
&\ =\  \lim_{m\to\infty} \Prob{\hat{F}^m_\bp(T_0) < s} \\
&\ =\  \Prob{F_\bp(T_0) < s}
  \label{eq:interchange-lim-prob} \\
&\ =\  \Prob{T_0 \dictorder F^{-1}_\bp(s)}
  \label{eq:invert-fp} \\
&\ =\  F_\bq(F^{-1}_\bp(s)). \label{eq:fq-equals-fp}
\end{align}
The interchange of the limit and the probability in
Eq.~\eqref{eq:interchange-lim-prob} follows from the bounded convergence
theorem, since $\hat{F}^m_\bp(T_0) \to F_\bp(T_0)$ a.s.\ and for all $m\ge 1$ we
have $\card{\hat{F}^m_\bp(T_0)} \le 1$ a.s.

Combining Eq.~\eqref{eq:rank-converge-distribution-uniform} and
Eq.~\eqref{eq:fq-equals-fp}, we see that
\begin{align*}
F_\bq(F^{-1}_\bp(s)) = s \implies F^{-1}_\bp(s) = F^{-1}_\bq(s),
\end{align*}
for $s\in [0,1]$.
Since $0 \le F_\bp(x,u) \le 1$, for each $(x,u)\in[0,1]^2$ we have
\begin{align*}
F^{-1}_\bq(F_\bp(x,u))
  &\ =\  F^{-1}_\bp(F_\bp(x,u)) \\
  &\ =\  F^{-1}_\bq(F_\bq(x,u)) \\
  &\ =\  (x,u).
\end{align*}

It follows that $F_\bp(x,u) = F_\bq(x,u)$ for all $(x, u) \in [0,1]^2$.
Fixing $u=0$, we obtain
\begin{align}
\tbp(x) = F_\bp(x,0) = F_\bq(x,0) = \tbq(x)
\label{eq:tbp-equals-tbq}
\end{align}
  for $x \in [0,1]$.

Assume, towards a contradiction, that $\bp \ne \bq$. Let $a$ be any rational
such that $\bp(a) \ne \bq(a)$, and suppose without loss of generality that
$\bq(a) < \bp(a)$. By the Cauchy criterion
(Remark~\ref{lem:intermediate-value-rats}), there exists some $b > a$ such that
\begin{align*}
\sum_{a < x < b}\bq(x) < \bp(a) - \bq(a).
\end{align*}
Then we have
\begin{align*}
\tbq(b)
  &\ =\ \tbq(a) + \bq(a) + \sum_{x \in \Rationals \cap (a, b)}\bq(x) \\
  &\ =\ \tbp(a) + \bq(a) + \sum_{x \in \Rationals \cap (a, b)}\bq(x) \\
  &\ <\ \tbp(a) + \bq(a) + (\bp(a) - \bq(a)) \\
  &\ =\ \tbp(a) +\bp(a)  \\
  &\ \le\ \tbp(b),
\end{align*}
and so $\tbp \ne \tbq$, contradicting Eq.~\eqref{eq:tbp-equals-tbq}.
\end{proof}

The following corollary is an immediate consequence.

\vspace*{5pt}
\begin{corollary}[Corollary~\ref{corr:p-ne-q-exists} in the main text]
\label{corr:p-ne-q-exists-appx}
If \mbox{$\bp \ne \bq$}, then there is some $m$ such that $R$ is not uniformly
distributed on $[m+1]$.
\end{corollary}

The next theorem strengthens Corollary~\ref{corr:p-ne-q-exists-appx} by showing
that $R$ is non-uniform for all but finitely many $m$.

\vspace*{5pt}
\begin{theorem}[Theorem~\ref{thm:p-ne-q-finitely-many-m} in the main text]
\label{thm:p-ne-q-finitely-many-m-appx}
If $\bp \ne \bq$, then there is some $M \ge 1$ such that for all $m \ge M$, the
rank $R$ is not uniformly distributed on $[m+1]$.
\end{theorem}

Before proving Theorem~\ref{thm:p-ne-q-finitely-many-m-appx},
we show the following lemma.

\begin{lemma}
\label{lem:bernoulli-exchangeable}
Suppose $Z_1, \dots, Z_{m+1}$ is a finitely exchangeable sequence of
Bernoulli random variables. If
\begin{align*}
S_m \defas \sum_{i=1}^{m}Z_i
\end{align*}
is not uniformly distributed on $[m+1]$, then
\begin{align*}
S_{m+1} \defas \sum_{i=1}^{m+1}Z_i
\end{align*}
is not uniformly distributed on $[m+2]$.
\end{lemma}

\begin{proof}
By finite exchangeability, there is some $r \in [0,1]$ such that the
distribution of every $Z_i$ is $\mathsf{Bernoulli}(r)$. There are two cases.

\paragraph{Case 1:} $r \ne 1/2$. For any $\ell\ge 1$, we have
\begin{align*}
\expect{S_{\ell}} = \expect{\sum_{i=1}^\ell Z_i}
  = \sum_{i=1}^\ell \expect{Z_i} = \ell r \ne r/2 = \expect{U_\ell},
\end{align*}
and so $S_\ell$ is not uniformly distributed on $[\ell+1]$. In particular, this
holds for $\ell$ equal to either $m$ or $m+1$, and so both the hypothesis and
conclusion are true.

\paragraph{Case 2:} $r = 1/2$. We prove the contrapositive. Suppose that
$S_{m+1}$ is uniformly distributed on $[m{+}1]$.

Assume $S_{m+1}$ is uniform and fix $k \in [m+1]$. By total probability, we have
\begin{align}
\begin{aligned}
\label{eq:total-probability-on-rm}
\Prob{S_m = k}
  &= \Prob{S_m = k \textrm{~and~} Z_{m+1} = 0} \\
  &\quad + \Prob{S_m = k \textrm{~and~} Z_{m+1} = 1}.
\end{aligned}
\end{align}
We consider the two events on the right-hand side of
Eq.~\eqref{eq:total-probability-on-rm} separately.

First, the event $\set{S_m = k}\,\cap\,\set{Z_{m+1} = 0}$ is the union over all
$\binom{m}{k}$ assignments of $(Z_1,\dots,Z_m)$ that have exactly $k$ ones and
$Z_{m+1} = 0$. All such assignments are disjoint events. Define the event
\begin{align*}
\begin{aligned}
A &\defas \{Z_1 = \cdots = Z_k = 1 \\
  &\quad \textrm{~and~} Z_{k+1} = \cdots = Z_m = Z_{m+1} = 0 \}.
\end{aligned}
\end{align*}
By finite exchangeability, each assignment has probability $\Prob{A}$, and so
\begin{align}
\label{eq:rm-eq-k-z0}
\Prob{S_m = k \textrm{~and~} Z_{m+1} = 0} = \binom{m}{k}\,\Prob{A}.
\end{align}
Now, observe that the event $\set{S_{m+1} = k}$ is the union of all
$\binom{m+1}{k}$ assignments of $(Z_1,\dots,Z_{m+1})$ that have exactly $k$
ones. All the assignments are disjoint events and each has probability
$\Prob{A}$, and so
\begin{align}
\label{eq:rmp1-eq-k}
\begin{aligned}
\Prob{S_{m+1}=k}
  &= \binom{m+1}{k}\,\Prob{A} \\
  &= \frac{1}{m+2}.
\end{aligned}
\end{align}

Second, the event $\set{S_m = k}\,\cap\,\set{Z_{m+1} = 1}$ is the union over all
$\binom{m}{k}$ assignments of $(Z_1,\dots,Z_m)$ that have exactly $k$ ones and
also $Z_{m+1} = 1$. All such assignments are disjoint events. Define the event
\begin{align*}
B &\defas \{Z_1 = \cdots = Z_{k} = Z_{m+1} = 1 \\
  &\quad \textrm{~and~} Z_{k+1} = \cdots = Z_m = 0\}.
\end{align*}
Again by finite exchangeability, each assignment has probability $\Prob{B}$,
and so
\begin{align}
\label{eq:rm-eq-k-z1}
\Prob{S_m = k \textrm{~and~} Z_{m+1} = 1} = \binom{m}{k}\,\Prob{B}.
\end{align}
Likewise, observe that the event $\set{S_{m+1} = k+1}$ is the union of all
$\binom{m+1}{k+1}$ assignments of \mbox{$(Z_1,\dots,Z_{m+1})$}
that have exactly $k+1$ ones.
All the assignments are disjoint events and each has probability $\Prob{B}$,
and so
\begin{align}
\label{eq:rmp1-eq-kp1}
\begin{aligned}
\Prob{S_{m+1}=k+1}
  &= \binom{m+1}{k+1}\,\Prob{B} \\
  &= \frac{1}{m+2}.
\end{aligned}
\end{align}

We now take Eq.~\eqref{eq:total-probability-on-rm}, divide by $1/(m+2)$,
and replace terms using Eqs.~\eqref{eq:rm-eq-k-z0}, \eqref{eq:rmp1-eq-k},
\eqref{eq:rm-eq-k-z1}, and \eqref{eq:rmp1-eq-kp1}:
\begin{align*}
&\frac{\Prob{S_m = k}}{1/(m+2)} \\
&\begin{aligned}
\quad &= \frac{\Prob{S_m = k \textrm{~and~} Z_{m+1} = 0}}{1/(m+2)} \\
  &\qquad + \frac{\Prob{S_m = k \textrm{~and~} Z_{m+1} = 1}}{1/(m+2)}
\end{aligned} \\
&\quad = \frac{\binom{m}{k}\Prob{A}}{\binom{m+1}{k}\Prob{A}}
  + \frac{\binom{m}{k}\Prob{B}}{\binom{m+1}{k+1}\Prob{B}} \\
&\begin{aligned}
\quad &= \frac{m!}{k!(m-k)!}\frac{k!(m+1-k)!}{(m+1)!} \\
  &\qquad + \frac{m!}{k!(m-k)!}\frac{(k+1)!(m+1-(k+1))!}{(m+1)!}\\
\end{aligned} \\
&\quad = \frac{m+1-k}{m+1} + \frac{k+1}{m+1} \\
&\quad = \frac{m+2}{m+1} \\
&\quad = \frac{1/(m+1)}{1/(m+2)},
\end{align*}
and so we conclude that $\Prob{S_m = k} = 1/(m+1).$
\end{proof}

We are now ready to prove Theorem~\ref{thm:p-ne-q-finitely-many-m-appx}.

\begin{proof}[Proof of Theorem~\ref{thm:p-ne-q-finitely-many-m-appx}]
Suppose $\bp \ne \bq$. By Corollary~\ref{corr:p-ne-q-exists-appx}, there is some
$M\ge 1$ such that the rank statistic
$R = \sum_{i=1}^{M}\indicator{T_i \prec T_0}$ for $m = M$ is
non-uniform over $[M+1]$.
Observe that the rank statistic for $m = M+1$ is given by
$\sum_{i=1}^{M+1}\indicator{T_i \prec T_0}$.

Now, each indicator $Z_i \defas \indicator{T_i \prec T_0}$ is a Bernoulli
random variable, and they are identically distributed since
$(T_1, \dots, T_{M+1})$
is an i.i.d.\ sequence. Furthermore the sequence $(Z_1,\dots,Z_{M+1})$ is
finitely exchangeable since the $Z_i$ are conditionally independent given $T_0$.
Then the sequence of indicators
$(\indicator{T_1 \prec T_0}, \indicator{T_2 \prec T_0},
  \dots, \indicator{T_{M+1} \prec T_0})$
satisfy the hypothesis of Lemma~\ref{lem:bernoulli-exchangeable}, and
so the rank statistic for $M+1$ is non-uniform. By induction, the rank
statistic is non-uniform for all $m \ge M$.
\end{proof}

In fact, unless $\bp$ and $\bq$ satisfy an adversarial symmetry relationship
under the selected ordering $\prec$, the rank is non-uniform for
\textit{any} choice of $m \ge 1$.
Let $\vlt$ denote the lexicographic order on $\dom \times [0,1]$ induced
by $(\dom, \prec)$ and $([0,1], <)$.

\vspace*{5pt}
\begin{corollary}[Corollary~\ref{corr:total-order-no-symmetry} in the main text]
\label{corr:total-order-no-symmetry-appx}
Suppose $\Prob{(X,U_1) \vlt (Y,U_0)} \ne 1/2$ for
$Y \sim \bq$, $X \sim \bp$, and $U_0, U_1 \simiid \mathsf{Uniform}(0,1)$.
Then for all $m \ge 1$, the rank $R$ is not uniformly distributed on $[m+1]$.
\end{corollary}

\begin{proof}
If $\Prob{{(X,U_1) \vlt (Y,U_0)}} \ne 1/2$ then $R$ is non-uniform for $m=1$.
The conclusion follows by Theorem~\ref{thm:p-ne-q-finitely-many-m-appx}.
\end{proof}

\subsection{An ordering that witnesses $\bp \neq \bq$ for $m=1$}

We now describe an ordering $\prec$ for which, when $m=1$, we have
$\Prob{R = 0} > 1/2$.

Define
\begin{align*}
A \defas \set{x \in \dom \mid \bq(x) > \bp(x)}
\end{align*}
to be the set of all elements of $\dom$ that have a greater probability according
to $\bq$ than according to $\bp$, and let $A^c$ denote its complement.
Let $\bh_{\bp, \bq}$ be the signed measure given by the difference
$\bh_{\bp, \bq}(x) \defas \bq(x) - \bp(x)$
between $\bq$ and $\bp$; for the rest of this subsection, we denote this
simply by $\bh$.
Let $\plt$ be any total order on $\dom$ satisfying
\begin{itemize}
\item if $\bh(x) > \bh(x') $ then $x \plt x'$; and
\item if $\bh(x) < \bh(x') $ then $x \pgt x'$.
\end{itemize}

The linear ordering $\plt$ may be defined arbitrarily for all pairs $x$ and $x'$
which satisfy $\bh(x) = \bh(x')$. As an immediate consequence, $x \plt x'$
whenever $x \in A$ and $x' \in A^c$. Intuitively, the ordering is designed to
ensure that elements $x \in A$ are ``small'', and are ordered by decreasing
value of $\bq(x) - \bp(x)$ (with ties broken arbitrarily); elements $x \in A^c$
are ``large'' and are ordered by increasing value of $\bp(x) - \bq(x)$ (again,
with ties broken arbitrarily). The smallest element in $\dom$ maximizes
$\bq(x) - \bp(x)$ and the largest element in $\dom$ maximizes $\bp(x)-\bq(x)$.

We first establish some easy lemmas.

\vspace*{5pt}
\begin{lemma}
$A = \emptyset$ if and only if $\bp = \bq$.
\end{lemma}

\begin{proof}
Immediate.
\end{proof}

\vspace*{5pt}
\begin{lemma}
\label{lem:sum-a-eq-sum-ac}
\begin{align*}
\sum_{x \in A}\left[\bq(x) - \bp(x)\right] = \sum_{x \in A^c}[\bp(x) - \bq(x)].
\end{align*}
\end{lemma}

\begin{proof}
We have
\begin{align*}
\begin{aligned}
&\sum_{x \in A}[\bq(x) - \bp(x)] - \sum_{x \in A^c}[\bp(x) - \bq(x)] \\
  &\quad = \sum_{x\in \dom}\bq(x) - \sum_{x\in \dom}\bp(x)  = 0,
\end{aligned}
\end{align*}
as desired.
\end{proof}

Given a probability distribution $\br$, define its cumulative distribution
function $\tbr$ by $\tbr(x) \defas \sum_{y \plt x}\br(y)$.

\vspace*{5pt}
\begin{lemma}
\label{lem:cdf-q-exceed-cdf-p}
$\tbq(x) > \tbp(x)$ for all $x \in \dom$.
\end{lemma}

\begin{proof}
Let $\dom_x \defas \set{y \in \dom \mid y \plt x}$. If $x \in A$ then
$\dom_x \subseteq A$, and so
\begin{align*}
\tbq(x) - \tbp(x) = \sum_{y \in \dom_x}[\bq(y) - \bp(y)]> 0,
\end{align*}
since all terms in the sum are positive.

Otherwise, $y\in A$ for all $y \plt x$, and so $A \subseteq \dom_x$. Let
$A^c_x \defas \set{y \in A^c \mid y \plt x}$. Then
\begin{align*}
&\tbq(x) - \tbp(x) \\
  &\ = \sum_{y \plt x}[\bq(y) - \bp(y)] \\
  &\ = \sum_{y \in A}[\bq(y) - \bp(y)] + \sum_{y \in A^c_x}[\bq(y) - \bp(y)] \\
  &\ = \sum_{y \in A_x}[\bq(y) - \bp(y)] - \sum_{y \in A^c_x}[\bp(y) - \bq(y)] \\
  &\ > \sum_{y \in A_x}[\bq(y) - \bp(y)] - \sum_{y \in A^c}[\bp(y) - \bq(y)] \\
  &\ = 0,
\end{align*}
establishing the lemma.
\end{proof}

We now analyze $\Prob{R = 0}$ in the case where $m=1$.
In this case, we may drop some subscripts and write $Y$ in place of $X_1$, so
that our setting reduces to the following random variables:
\begin{align*}
X_\bp &\sim \bp \\
Y_\bq &\sim \bq \\
R_{\bp, \bq} \mid X_\bp, Y_\bq &\sim \begin{cases}
    0 & \mbox{if } X_\bp \pgt Y_\bq, \\
    1 & \mbox{if } X_\bp \plt Y_\bq, \\
    \mathrm{Bernoulli}(1/2) & \mbox{if } X_\bp = Y_\bq. \\
  \end{cases}
\end{align*}
(We have indicated $\bp$ and $\bq$ in the subscripts, for use in the
next subsection.)

In other words, the procedure samples $X_\bp \sim$ $\bp$ and $Y_\bq \sim \bq$
independently.
Given these values, it then sets $R_{\bp, \bq}$ to be 0 if $X_\bp \pgt Y_\bq$,
to be 1 if $X_\bp \plt Y_\bq$,
and the outcome of an independent fair coin flip otherwise.

For the rest of this subsection, we will refer to these random variables simply
as $X$, $Y$, and $R$, though later on we will need them for several choices of
distributions $\bp$ and $\bq$ (and accordingly will retain the subscripts).

We now prove the following theorem.

\vspace*{5pt}
\begin{theorem}[Theorem 3.6 in the main text]
\label{thm:prob-y-x-half}
If $\bp \ne \bq$, then for $m=1$ and the ordering $\prec$ defined above, we have
$\Prob{R = 0} > 1/2$.
\end{theorem}

\begin{proof}
From total probability and independence of $X$ and $Y$, we have
\begin{align*}
&\Prob{R = 0} \\
&= \sum_{x,y \in \dom} \Prob{R {=} 0\,{\mid}\, X {=} x, Y {=} y} \Prob{Y=y}\Prob{X=x} \\
&= \sum_{x,y \in \dom} \Prob{R {=} 0\,{\mid}\, X {=} x, Y {=} y} \bq(y)\bp(x) \\
&= \sum_{x \in \dom} \Prob{R {=} 0\,{\mid}\, X {=} x, Y {=} x} \bq(x)\bp(x) \\
  &\qquad + \sum_{y \plt x \in \dom} \Prob{R {=} 0\,{\mid}\, X {=} x, Y {=} y} \bq(y)\bp(x) \\
  &\qquad + \sum_{x \plt y \in \dom} \Prob{R {=} 0\,{\mid}\, X {=} x, Y {=} y} \bq(y)\bp(x) \\
&\begin{aligned}
  &=\frac{1}{2} \sum_{x \in \dom} \bq(x)\bp(x)
    + 1 \sum_{y \plt x \in \dom} \bq(y)\bp(x) \\
  &\qquad + 0 \sum_{x \plt y \in \dom} \bq(y)\bp(x)
\end{aligned}\\
&= \frac{1}{2} \sum_{x \in \dom} \bp(x)\bq(x) + \sum_{x \in \dom}\tbq(x)\bp(x).
\end{align*}

An identical argument establishes that
\begin{align*}
\Prob{R {=} 1}
  &= \frac{1}{2} \sum_{x \in \dom} \bp(x)\bq(x) + \sum_{x \in \dom}\tbp(x)\bq(x).
\end{align*}
Since $\Prob{R {=} 0} + \Prob{R = 1} = 1$, it suffices to establish that
$\Prob{R=0} > \Prob{R=1}$. We have
\begin{align*}
&\Prob{R = 0} - \Prob{R=1}\\
&= \sum_{x \in \dom}\tbq(x)\bp(x) - \sum_{x \in \dom}\tbp(x)\bq(x) \\
&> \sum_{x \in \dom}\tbp(x)\bp(x) - \sum_{x \in \dom}\tbp(x)\bq(x) \\
&= \sum_{x \in \dom}\tbp(x)[\bp(x) - \bq(x)] \\
&= \sum_{x \in A^c}\tbp(x)[\bp(x) - \bq(x)]
  - \sum_{x \in A}\tbp(x)[\bq(x) - \bp(x)] \\
&\begin{aligned}
  &\ge \sum_{x \in A^c}\bigl(\max_{y \in A}\tbp(y)\bigr)[\bp(x) - \bq(x)] \\
  &\qquad - \sum_{x \in A}\tbp(x)[\bq(x) - \bp(x)]
\end{aligned}\\
&\begin{aligned}
  &= \sum_{x \in A} \bigl(\max_{y \in A}\tbp(y)\bigr) [\bq(x) - \bp(x)]\\
  &\qquad - \sum_{x \in A}\tbp(x)[\bq(x) - \bp(x)]
\end{aligned}\\
  &= \sum_{x \in A}\bigl(\max_{y \in A}\tbp(y)-\tbp(x)\bigr)[\bq(x) - \bp(x)] \\
  &> 0.
\end{align*}
The first inequality follows from Lemma~\ref{lem:cdf-q-exceed-cdf-p};
the second inequality follows from monotonicity of $\tbp$;
the second-to-last equality follows from Lemma~\ref{lem:sum-a-eq-sum-ac};
and the final inequality follows from the fact that all terms in the
sum are positive.
\end{proof}

\subsection{A tighter bound in terms of $L_\infty(\bp, \bq)$}

We have just exhibited an ordering such that when
$\bp \ne \bq$ and $m=1$,
we have $\Prob{R=0} > 1/2$. We are now
interested in obtaining a tighter lower bound on this
probability in terms of
the $L_\infty$ distance between $\bp$ and $\bq$.

In this subsection and the following one, we assume that $\dom$ is finite.
We first note the following immediate lemma.

\vspace*{5pt}
\begin{lemma}
\label{lem:L_infty-continuous-of-R0}
Let $B, C \subseteq \dom$.
For all $\bp, \bq$ and all $\delta > 0$
there is an $\epsilon > 0$
such that for all distributions $\bp'$ on $\dom$ with
$\sup_{x\in\dom} |\bp(x) - \bp'(x)| < \epsilon$, we have
\begin{align*}
\begin{aligned}
&\bigl|\Pr(R_{\bp, \bq} = 0\given X_\bp \in B,\ Y_\bq \in C) \\
  &\quad - \Pr(R_{\bp', \bq} = 0\given X_{\bp'} \in B,\ Y_\bq \in C)\bigr|
  < \delta.
\end{aligned}
\end{align*}
\end{lemma}

\vspace*{5pt}
\begin{definition}
\label{def:epsilon-discrete}
We say that $\bp$ is \defn{$\epsilon$-discrete} (with respect to $\bq$) if for
all $a, b\in \dom$ we have
\begin{align*}
\bigl|\bh_{\bp, \bq}(a) - \bh_{\bp, \bq}(b)\bigr| \geq \epsilon.
\end{align*}
\end{definition}

From Lemma~\ref{lem:L_infty-continuous-of-R0}
we immediately obtain the following.
\begin{lemma}
\label{lem:change-to-e-discrete-pair}
For all $\bp, \bq$ and all $\delta >0$ there is an $\epsilon > 0$ and an
$\epsilon$-discrete distribution $\bp_\epsilon$ on $\dom$ such that for
all $B, C \subseteq \dom$,
\begin{align*}
\begin{aligned}
&\bigl|\Pr(R_{\bp, \bq} = 0 \given X_\bp \in B,\ Y_\bq \in C) \\
  &\quad - \Pr(R_{\bp_\epsilon, \bq} = 0 \given X_{\bp_\epsilon} \in B,\ Y_\bq \in C)\bigr|
  < \delta.
\end{aligned}
\end{align*}
\end{lemma}

The next lemma will be crucial for proving our bound.

\vspace*{5pt}
\begin{lemma}
\label{lem:mass-move-p}
Let $\bp_0$ and $\bp_1$ be probability measures on $\dom$, and let
$\vlt$ be a total order on $\dom$ such that
if $\bh_{\bp_0, \bq}(x) > \bh_{\bp_0, \bq}(x')$  then $x \vlt x'$ and
if $\bh_{\bp_0, \bq}(x) < \bh_{\bp_0, \bq}(x')$  then $x \vgt x'$.
Suppose that
if $\bh_{\bp_0, \bp_1}(x) > 0$ and $\bh_{\bp_0, \bp_1}(y) \leq 0$, then $x \vlt y$.
Then $\Pr(R_{\bp_0, \bq} = 0) \geq \Pr(R_{\bp_1, \bq} = 0)$.
\end{lemma}
\begin{proof}
Note that
\begin{align*}
&\Pr(R_{\bp_1, \bq} = 0 \given Y_\bq = y)\\
&= \sum_{x \vgt y}\bp_1(x) +\frac{1}{2}\bp_1(y) \\
&= \sum_{x \vgt y}\bp_0(x) + \bh_{\bp_0, \bp_1}(x) +\frac{1}{2}[\bp_0(y) + \bh_{\bp_0, \bp_1}(y)] \\
&= \Pr(R_{\bp_0, \bq} = 0 \given Y_\bq = y) + \sum_{x \vgt y}\bh_{\bp_0, \bp_1}(x) +\frac{1}{2}\bh_{\bp_0, \bp_1}(y) \\
&= \Pr(R_{\bp_0, \bq} = 0 \given Y_\bq = y) - \sum_{x \vlt y}\bh_{\bp_0, \bp_1}(x) -\frac{1}{2}\bh_{\bp_0, \bp_1}(y),
\end{align*}
where the last equality holds because $\sum_{x\in \dom} \bh_{\bp_0, \bp_1}(x) = 0$.
But by our assumption, we know that
$\sum_{x \vlt y}\bh_{\bp_0, \bp_1}(x) +\frac{1}{2}\bh_{\bp_0, \bp_1}(y)$
is non-negative and so
$\Pr(R_{\bp_1, \bq} = 0 \given Y_\bq = y) \leq \Pr(R_{\bp_0, \bq} = 0 \given Y_\bq = y)$,
from which the result follows.
\end{proof}

We will now provide a lower bound on $\Pr(R_{\bp, \bq} = 0)$.

\vspace*{5pt}
\begin{proposition}
\label{prop:lower-bound-pAh}
\begin{equation}
\Pr(R_{\bp, \bq} = 0)
  \geq  \frac{1}{2} + \frac{1}{2}\max_{x \in \dom} \bh_{\bp, \bq}(x)^2.
\label{eq:lower-bound-pAh}
\end{equation}
\end{proposition}

\begin{proof}
Recall that $A \defas \set{x \in \dom \mid \bq(x) > \bp(x)}$.
First note that by Lemma~\ref{lem:L_infty-continuous-of-R0}, we may assume
without loss of generality that $|A| = |\dom \setminus A|$, by adding elements
of mass arbitrarily close to $0$. Let $k\defas |A|$. Further, by
Lemma~\ref{lem:change-to-e-discrete-pair} we may assume without loss of
generality that $\bp, \bq$ are an $\epsilon$-discrete pair (for some fixed but
small $\epsilon$) with $|\dom| \cdot \epsilon < L_{\infty}(\bp, \bq)$. Let
$(x_0^+, \dots, x_{k-1}^+)$ be the collection $A$ listed in $\prec$-increasing
order.  Let $(x_0^-, \dots, x_{k-1}^-)$ be the collection $\dom \setminus A$
listed in $\prec$-increasing order.

Let $\bp^*$ be any probability measure such that
\begin{equation*}
\bp^*(x) =
\begin{cases}
\bp(x) -e(\ell) & (x = x^-_\ell; e(\ell) \geq 0),\\
\bq(x) - (k-\ell) \cdot \epsilon & (x = x^+_\ell; 0 \leq \ell < k-1),\\
\bp(x) & (x = x^+_{0}).
\end{cases}
\end{equation*}
Note that for all $x, y \in \dom$, we have $y \prec x$ if and only if
$\bh_{\bp^*, \bq}(x) < \bh_{\bp^*, \bq}(y)$.

Now, for every $\ell < k-1$ we have $\bh_{\bp, \bq}(x_\ell^+) \geq \ell \cdot
\epsilon$ (as $\bp, \bq$ are an $\epsilon$-discrete pair), and so we can always
find such a $\bp^*$. In particular the following are immediate.

\begin{itemize}
\item[(a)] $x \prec y$ if and only if $\bh_{\bp^*, \bq}(x) > \bh_{\bp^*, \bq}(y)$,

\item[(b)] $\bh_{\bp, \bq}(x^+_{0}) = \bh_{\bp^*, \bq}(x^+_{0})$,

\item[(c)] if $\bh_{\bp, \bq^*}(x) > 0$ and $\bh_{\bp, \bp^*}(y) \leq 0$ then $x \prec y$, and

\item[(d)] $(\bp, \bq^*)$ is an $\epsilon$-discrete pair.
\end{itemize}

Note that $\Pr(R_{\bp, \bq}   = 0) \geq \Pr(R_{\bp^*, \bq}   = 0)$,
by Lemma~\ref{lem:mass-move-p} and (c).
For simplicity, let $A_0 \defas \{x_{0}^+\}$,
$A_1 \defas \{x_i^+\}_{1 \leq i \leq k-1}$ and $D \defas \dom \setminus A$.

We now condition on the value of $Y_{\bq}$, in order to calculate $\Pr(R_{\bp^*, \bq} = 0)$.

\paragraph{Case 1:} $Y_{\bq} = x_i^-$. We have
\begin{align*}
\Pr(R_{\bp^*, \bq} \,{=}\, 0 {\given} Y_{\bq} \,{=}\, x_i^-)
  \,{=}\, \sum_{i <\ell < k} \bp^*(x_\ell^-) + \frac{1}{2} \bp^*(x_i^-).
\end{align*}

\paragraph{Case 2:} $Y_{\bq} \in A_1$. We have
\begin{align*}
\Pr(R_{\bp^*, \bq} \,{=}\,0 {\given} Y_{\bq} \,{\in}\, A_1)
  \,{=}\, \bp^*(D) + \frac{1}{2} \bp^*(A_1) + f_0(\epsilon),
\end{align*}
where $f_0$ is a function satisfying $\lim_{\epsilon \to 0}f_0(\epsilon) = 0$.

\paragraph{Case 3:} $Y_{\bq} \in A_0$. We have
\begin{align*}
\Pr(R_{\bp^*, \bq}\,{=}\,0 {\given} Y_{\bq} \,{\in}\, A_0)
  \,{=}\,\bp^*(A_1) + \bp^*(D) + \frac{1}{2} \bp^*(A_0).
\end{align*}

We may calculate these terms as follows:
\begin{align*}
\bp^*(D) &= \bq(D) + \bh_{\bp, \bq}(x_{0}^+) + ({k(k-1)}/{2}) \epsilon,\\
\bp^*(A_1) &= \bq(A_1) - ({k(k-1)}/{2}) \epsilon, \\
\bp^*(A_0) &= \bq(A_0) - \bh_{\bp, \bq}(x^+_{0}).
\end{align*}
Putting all of this together, we obtain
\begin{align*}
&\Pr(R_{\bp^*, \bq}   = 0) \\
&\begin{aligned}
&= \sum_{i < k}\sum_{i < \ell < k}\bq(x_i^-) \bp^*(x_\ell^-) + \frac{1}{2} \sum_{i < k}\bq(x_i^-) \bp^*(x_i^-) \\
& \quad + \bq(A_1) \bp^*(D) + \frac{1}{2} \bq(A_1) \bp^*(A_1) + \bq(A_1)  f_0(\epsilon) \\
& \quad + \bq(A_0) \bp^*(A_1) + \bq(A_0) \bp^*(D) + \frac{1}{2} \bq(A_0) \bp^*(A_0) \\
\end{aligned}\\
&\begin{aligned}
&= \sum_{i < k}\sum_{i < \ell < k}\bq(x_i^-) [\bq(x_\ell^-) - \bh_{\bp^*, \bq}(x_\ell^-)]  \\
&\quad + \frac{1}{2} \sum_{i<k}\bq(x_i^-) [\bq(x_i^-) - \bh_{\bp^*, \bq}(x_i^-)] \\
&\quad + \bq(A_1) [\bq(D) + \bh_{\bp, \bq}(x_{0}^+)] + \frac{1}{2} \bq(A_1) \bq(A_1)  \\
&\quad + \bq(A_0) \bq(A_1) + \bq(A_0) [\bq(D) + \bh_{\bp, \bq}(x_{0}^+)] \\
&\quad + \frac{1}{2} \bq(A_0) [\bq(A_0) - \bh_{\bp, \bq}(x^+_{0})] +  f_1(\epsilon) \\
\end{aligned}\\
&\begin{aligned}
&=  \sum_{i < k}\sum_{i < \ell < k}\bq(x_i^-)\bq(x_\ell^-)  + \frac{1}{2} \sum_{i < k}\bq(x_i^-) \bq(x_i^-) \\
&\quad + \bq(A_1) \bq(D) + \frac{1}{2} \bq(A_1) \bq(A_1) + \bq(A_0) \bq(A_1) \\
&\quad + \bq(A_0) \bq(D) + \frac{1}{2} \bq(A_0) \bq(A_0) \\
&\quad - \sum_{i < k}\sum_{i < \ell < k} \bq(x_i^-)\bh_{\bp^*, \bq}(x_\ell^-) \\
&\quad - \frac{1}{2} \sum_{i < k}\bq(x_i^-) \bh_{\bp^*, \bq}(x_i^-) + \bq(A_1) \bh_{\bp, \bq}(x_{0}^+)  \\
&\quad + \bq(A_0)  \bh_{\bp, \bq}(x_{0}^+) - \frac{1}{2} \bq(A_0)  \bh_{\bp, \bq}(x^+_{0}) +  f_1(\epsilon)
\end{aligned} \\
&\begin{aligned}
&= \sum_{i < k}\sum_{i < \ell < k}\bq(x_i^-)\bq(x_\ell^-)  + \frac{1}{2} \sum_{i < k}\bq(x_i^-) \bq(x_i^-) \\
&\quad + \bq(A_1) \bq(D) + \frac{1}{2} \bq(A_1) \bq(A_1) + \bq(A_0) \bq(A_1) \\
&\quad + \bq(A_0) \bq(D) + \frac{1}{2} \bq(A_0) \bq(A_0) \\
&\quad - \sum_{i < k}\sum_{i < \ell < k} \bq(x_i^-)\bh_{\bp^*, \bq}(x_\ell^-) \\
&\quad - \frac{1}{2} \sum_{i < k}\bq(x_i^-) \bh_{\bp^*, \bq}(x_i^-) + \bq(A_1) \bh_{\bp, \bq}(x_{0}^+) \\
&\quad + \frac{1}{2} \bq(A_0)  \bh_{\bp, \bq}(x^+_{0}) +  f_1(\epsilon),
\end{aligned}
\end{align*}
where $f_1$ is a function satisfying $\lim_{\epsilon \to 0} f_1(\epsilon) = 0$.

We also have
\begin{align*}
\frac{1}{2} &=   \Pr(R_{\bq, \bq}   = 0) \\
&\begin{aligned}
&= \sum_{i < k}\sum_{i < \ell < k}\bq(x_i^-)\bq(x_\ell^-) + \frac{1}{2} \sum_{i < k}\bq(x_i^-) \bq(x_i^-) \\
&\quad + \bq(A_1) \bq(D) + \frac{1}{2} \bq(A_1) \bq(A_1)\\
&\quad + \bq(A_0) \bq(A_1) + \bq(A_0) \bq(D) + \frac{1}{2} \bq(A_0) \bq(A_0).
\end{aligned}
\end{align*}

Putting these two equations together, we obtain
\begin{align*}
&\Pr(R_{\bp^*, \bq} = 0) - \frac{1}{2} \\
&= \Pr(R_{\bp^*, \bq}   = 0) - \Pr(R_{\bq, \bq}   = 0) \\
&\begin{aligned}
&= - \sum_{i < k}\sum_{i < \ell < k} \bq(x_i^-) \bh_{\bp^*, \bq}(x_\ell^-) \\
& \quad - \frac{1}{2} \sum_{i < k}\bq(x_i^-)  \bh_{\bp^*, \bq}(x_i^-) + \bq(A_1) \bh_{\bp, \bq}(x_{0}^+) \\
& \quad  + \frac{1}{2} \bq(A_0)   \bh_{\bp, \bq}(x^+_{0}) +  f_1(\epsilon)
\end{aligned}\\
&\geq \frac{1}{2} \bq(A_0)   \bh_{\bp, \bq}(x^+_{0}) +  f_1(\epsilon),
\end{align*}
as $ \bh_{\bp^*, \bq}(x_\ell^-) \leq 0$ for all $\ell < k$ and $ \bh_{\bp, \bq}(x^+_{0}) \leq 0$.

But we know that
\begin{align*}
\bq(A_0) = \bq(x^+_{0}) = \bp^*(x^+_{0}) + \bh_{\bp, \bq}(x^+_{0}) \geq \bh_{\bp, \bq}(x^+_{0}).
\end{align*}
Therefore, as $\bh_{\bp, \bq}(x^+_{0})$ is the maximal value of $\bh_{\bp, \bq}$,
by taking the limit as $\epsilon\to 0$ we obtain
\begin{align*}
\Pr(R_{\bp^*, \bq}   = 0) &\geq  \frac{1}{2} + \frac{1}{2}\max_{x \in \dom} \bh_{\bp, \bq}(x)^2,
\end{align*}
as desired.
\end{proof}

Finally, we arrive at the following theorem.

\vspace*{5pt}
\begin{theorem}
\label{thm:lower-bound-linf}
Given probability measure $\bp, \bq$ on $\dom$ there is a linear ordering
$\sqsubset$ of $\dom$ such that if $X_\bp$ and $Y_\bq$ are sampled independently
from $\bp$ and $\bq$ respectively then
\begin{equation}
\Pr(X_\bq \sqsubset Y_\bp) \geq  \frac{1}{2} + \frac{1}{2}L_{\infty}(\bp, \bq)^2.
\label{eq:lower-bound-linf}
\end{equation}
\end{theorem}

\begin{proof}
Note that
\begin{equation*}
L_{\infty}(\bp, \bq) = \max\{ \max_{x \in \dom}\bh_{\bp, \bq}(x),\, \max_{x \in \dom} \bh_{\bq, \bp}(x)\}.
\end{equation*}

If $L_{\infty}(\bp, \bq) = \max_{x \in \dom}\bh_{\bp, \bq}(x)$, then the theorem
follows from Proposition~\ref{prop:lower-bound-pAh} using the ordering
$x \sqsubset y$ if and only if $\bh_{\bp, \bq}(x) > \bh_{\bp, \bq}(y)$.

If, however, $L_{\infty}(\bp, \bq) = \max_{x \in \dom}\bh_{\bq, \bp}(x)$, then
the theorem follows from Proposition~\ref{prop:lower-bound-pAh} by interchanging
$\bp$ and $\bq$, i.e., by using the ordering $x \sqsubset y$ if and only if
$\bh_{\bq, \bp}(x) > \bh_{\bq, \bp}(y)$.
\end{proof}

\subsection{Sample complexity}

We now show how to amplify this result by repeated trials to obtain a
bound on the sample complexity of the main algorithm for determining
whether $\bp = \bq$.

Let $\sqsubset$ be the linear ordering defined in
Theorem~\ref{thm:lower-bound-linf}.

\begin{theorem}[Theorem~\ref{thm:sample-complexity} in the main text]
\label{thm:sample-complexity-appx}
Given significance level $\alpha = 2\Phi(-c)$ for $c > 0$,
the proposed test with ordering $\sqsubset$ and $m=1$ achieves power
$\beta \ge 1 - \Phi(-c)$ using
\begin{equation}
n \approx 4c^2/L_\infty(\bp, \bq)^4
\label{eq:sample-complexity-appx}
\end{equation}
samples from $\bq$, where $\Phi$ is the cumulative distribution function of a
standard normal.
\end{theorem}

\begin{proof}
Assume without loss of generality that the order $\sqsubset$ from
Theorem~\ref{thm:lower-bound-linf} is such that
$L_\infty = \max_{s \in \dom}(\bq(x) - \bp(x))$.
Let $(Y_1,\dots,Y_n) \simiid \bq$ be the $n$ samples from $\bq$. With $m=1$, the
testing procedure generates $n$ samples $(X_1,\dots,X_n) \simiid \bp$, and $2n$
uniform random variables $(U^Y_{1}, \dots, U^Y_{n}, U^X_{1}, \dots, U^X_{n})
\simiid \mathsf{Uniform}(0,1)$ to break ties.
Let $\vlt$ denote the lexicographic order on $\dom \times [0,1]$ induced
by $(\dom, \vlt)$ and $([0,1], <)$.
Define $W_i \defas \indicator{(Y_i, U^Y_i) \vlt (X_i, U^X_i)}$, for $1
\le i \le n$, to be the rank of the $i$-{th} observation from
$\bq$.

Under the null hypothesis $\Hnull$, each rank $W_i$ has
distribution $\mathsf{Bernoulli}(1/2)$ by Lemma~\ref{lem:uniformity-rank-statistics}.
Testing for uniformity of the ranks on $\set{0, 1}$ is equivalent to testing
whether a coin is unbiased given the i.i.d.\ flips $\set{W_1,\dots,W_n}$. Let
$\hat{B} \defas \sum_{i=1}^{n} (1-W_i)/ n$ denote the empirical proportion of
zeros.
By the central limit theorem, for sufficiently large $n$, we have that
$\hat{B}$ is approximately normally distributed with mean $1/2$ and standard
deviation $1/(2\sqrt{n})$.
For the given significance level $\alpha = 2\Phi(-c)$, we form the two-sided
$\reject$ region $F = (-\infty, \gamma) \cup (\gamma, \infty)$,
where the critical value $\gamma$ satisfies
\begin{align}
\label{eq:critical-value}
c = \frac{\gamma - 1/2}{1/(2\sqrt{n})} = 2\sqrt{n}(\gamma - 1/2).
\end{align}
Replacing $n$ in Eq.~\eqref{eq:sample-complexity}, we obtain
\begin{align}
\gamma
  &= 1/2 + c/(2\sqrt{n}) \notag \\
  &= 1/2 + c/(2(2c/L_{\infty}(\bp, \bq)^2)) \notag \\
  &= 1/2 + L_{\infty}(\bp,\bq)^2/4.
  \label{eq:gamma-in-terms-of-c}
\end{align}
This construction ensures that $\Prob{ \reject \mid \Hnull} = \alpha$.

We now show that the test with this rejection region has power $\beta \ge
\Prob{\reject \mid \Halt} = 1 - \Phi(-c)$.
Under the alternative hypothesis $\Halt$, each $W_i$ has (in the worst case)
distribution $\mathsf{Bernoulli}(1/2+L_{\infty}(\bp,\bq)^2/2)$ by
Theorem~\ref{thm:lower-bound-linf}, so that the empirical proportion $\hat{B}$
is approximately normally distributed with mean at least
$1/2+L_{\infty}(\bp,\bq)^2/2$ and standard deviation at most $1/(2\sqrt{n})$.
Under the alternative distribution of $\hat{B}$, the standard score $c'$ of the
critical value $\gamma$ is
\begin{align}
c'
  &= \frac{\gamma - (1/2+L_{\infty}(\bp,\bq)^2/2)}{1/(2\sqrt{n})} \notag \\
  &= 2\sqrt{n}((1/2 + L_{\infty}(\bp,\bq)^2/4) - (1/2+L_{\infty}(\bp,\bq)^2/2)) \notag \\
  &= -2\sqrt{n}(L_{\infty}(\bp,\bq)^2/4) \notag \\
  &= -\sqrt{n}  L_{\infty}(\bp,\bq)^2/2 \notag \\
  &= -c \label{eq:bhat-alt-critical-value},
\end{align}
where the second equality follows from Eq.~\eqref{eq:gamma-in-terms-of-c}.
Observe that the $\notreject$ region
$F^c = [-\gamma, \gamma] \subset (-\infty, \gamma]$,
and so the probability that $\hat{B}$ falls in $F^c$ is at most the probability
that $\hat{B} < \gamma$, which by Eq.~\eqref{eq:bhat-alt-critical-value} is
equal to $\Phi(-c)$.
It is then immediate that $\beta \ge 1 - \Phi(-c)$.
\end{proof}

The following corollary follows directly from
Theorem~\ref{thm:sample-complexity}.

\vspace*{5pt}
\begin{corollary}
\label{corr:sample-complexity-error}
As the significance level $\alpha$ varies, the proposed test
with ordering $\sqsubset$ and $m=1$
achieves an overall error $(\alpha +
(1-\beta))/2 \le 3\Phi(-c)/2$ using $n = 4c^2/L_\infty(\bp, \bq)^4$
samples.
\end{corollary}

\subsection{Distribution of the test statistic under the alternative hypothesis}
\label{subsec:distribution-alternative}

In this subsection we derive the distribution of $R$ under the alternative hypothesis
$\bp \ne \bq$.
As before, write $\tbp(x) \defas \sum_{x' < x}\bp(x)$.

\vspace*{5pt}
\begin{theorem}
\label{thm:dist-r0-combinatorial-appx}
The distribution of $R$ is given by
\begin{equation}
\Prob{R = r} = \sum_{x \in \dom}H(x, m, r) \,\bq(x)
\label{eq:dist-r0-combinatorial-appx}
\end{equation}
for $0 \le r \le m$,
where $H(x,m,r) \defas$
\begin{align*}
~~~\begin{cases}
  \displaystyle\binompdf{r}{m}{\tbp(x)}
    & (\bp(x) = 0) \\[10pt]
  \displaystyle\frac{1}{m+1}
    & (\bp(x) = 1) \\[10pt]
  \displaystyle\sum_{e=0}^{m} \Bigg\{
    \Bigg[ \sum_{j=0}^{e}
    \displaystyle\binom{m-e}{r-j}
      \left[ \frac{\tbp(x)}{1-\bp(x)} \right]^{r-j}
      \span \notag \\
    \quad \displaystyle\left[ 1-\frac{\tbp(x)}{1-\bp(x)} \right]^{(m-e)-(r-j)}
      \left(\frac{1}{e+1}\right) \Bigg] \span \notag \\
    \quad \displaystyle\binompdf{m}{e}{\bp(x)} \Bigg\}
    & (0 < \bp(x) < 1)
  \end{cases}
\end{align*}
\end{theorem}

\begin{proof}
Define the following random variables:
\begin{align}
L &\defas \sum_{i=1}^{m} \mathbb{I}\left[X_i \plt X_0 \right], \\
E &\defas \sum_{i=1}^{m} \mathbb{I}\left[X_i = X_0 \right], \\
G &\defas \sum_{i=1}^{m} \mathbb{I}\left[X_i \pgt X_0 \right].
\end{align}
We refer to $L$, $E$, and $G$ as ``bins'', where $L$ is the ``less than'' bin,
$E$ is the ``equal to'' bin, and $G$ is the ``greater than'' bin (all with
respect to $X_0$). Total probability gives
\begin{align*}
\Prob{R = r}
  &= \sum_{x \in \dom}\Prob{R = r , X_0 = x} \\
  &= \sum_{\substack{x \in \dom \\ \bq(x) > 0}}\Prob{R = r \given X_0 = x} \bq(x).
\end{align*}

Fix $x \in \dom$ such that $\bq(x) > 0$. Consider $\Prob{R = r \given X_0 = s}$.
The counts in bins $L$, $E$, and $G$ are binomial random variables with $m$
trials, where the bin $L$ has success probability $\tbp(x)$, the bin $E$ has
success probability $\bp(x)$, and the bin $G$ has success probability
$1-(\tbp(x) + \bp(x))$.
We now consider three cases.

\paragraph{Case 1:} $\bp(x) = 0$. The event $\set{E=0}$ occurs with
probability one since each $X_i$, for $1 \le i \le m$, cannot possibly be equal
to $x$. Therefore, conditioned on $\set{X_0 = x}$, the event $\set{R = r}$
occurs if and only if $\set{L = r}$. Since $L$ is binomially distributed,
\begin{align*}
\Prob{R = r \given X_0 = x}
  &= \Prob{L=r \given X_0 = x} \\
  &= \binom{m}{r}\left[\tbp(x)\right]^r
    \left[1-\tbp(x)\right]^{m-r}.
\end{align*}

\paragraph{Case 2:} $\bp(x) = 1$. Then the event $\set{E=m}$ occurs with
probability one since each $X_i$, for $1 \le i \le m$, can only equal $s$. The
uniform numbers $U_0, \dots, U_m$ used to break the ties will determine the rank
$R$ of $X_0$. Let $B$ be the rank of $U_0$ among the $m$ other uniform random
variables $U_1, \dots, U_m$. The event $\set{R = r}$ occurs if and only if
$\set{B = r}$. Since the $U_i$ are i.i.d., $B$ is uniformly distributed over
$\set{0,1,2,\dots, m}$ by Lemma~\ref{lem:uniformity-rank-statistics}. Hence
\begin{align*}
\Prob{R = r \given X_0 = x}
  = \Prob{B=r \given X_0 = x}
  = \frac{1}{m+1}.
\end{align*}

\paragraph{Case 3:} $0 < \bp(x) < 1$. By total probability,
\begin{align*}
&\Prob{R = r \given X_0 = x} \\
&= \sum_{e=0}^{m} \Prob{R = r \given X_0 = x, E=e} \Prob{E=e \given X_0 = x}.
\end{align*}
Since $E$ is binomially distributed,
\begin{align*}
\Prob{E=e \given X_0 = x}
&= \binom{m}{e}\left[\bp(x)\right]^e \left[1-\bp(x)\right]^{m-e}.
\end{align*}

We now tackle the event $\set{R = r \given X_0 = x, E=e}$. The uniform numbers
$U_0, \dots, U_m$ used to break the ties will determine the rank $R$ of $X_0$.
Define $B$ to be the rank of $U_0$ among the $e$ other uniform random variables
assigned to bin $E$, i.e., those
$U_i$ for $1 \le i \le m$ such that $X_i = s$.  The random
variable $B$ is independent of all the $X_i$, but is dependent on $E$. Given
$\set{E=e}$, $B$ is uniformly distributed on $\set{0,1,\dots,e}$. By total
probability,
\begin{align*}
&\Prob{R = r \given X_0 = x, E=e}\\
&\begin{aligned}
&= \sum_{b=0}^{e} \big[ \Prob{R = r \given X_0 = x, E=e, B=b} \\
&\qquad \qquad \qquad \Prob{B=b \given E = e} \big] \\
\end{aligned}\\
&= \sum_{b=0}^{e} \Prob{R = r \given X_0 = x, E=e, B=b} \frac{1}{e+1}.
\end{align*}

Conditioned on $\set{E=e}$ and $\set{B=0}$, the event $\set{R = r}$ occurs if
and only if $\set{L = r}$, since exactly $0$ random variables in bin $E$ ``are
less'' than $X_0$, so exactly $r$ random variables in bin $L$ are needed to
ensure that the rank of $X_0$ is $r$. By the same reasoning, for $0 \le b \le e$,
conditioned on
$\set{E=e, B=b}$ we have $\set{R = r}$ if and only if $\set{L = r - b}$.

Now, conditioned on $\set{E=e}$, there are $m-e$ remaining assignments to be
split among bins $L$ and $G$. Let $i$ be such that $X_i \ne x$. Then the
relative probability that $X_i$ is assigned to bin $L$ is $\tbp(x)$ and to
bin $G$ is $1-(\tbp(x) + \bp(x))$. Renormalizing these probabilities, we
conclude that $L$ is conditionally (given $\set{E=e}$) a binomial random
variable with $m-e$ trials and success probability
$\tbp(x) / (\tbp(x) + (1-(\tbp(x) + \bp(x))))
= \tbp(x) / (1-\bp(x))$.
Hence
\begin{align*}
&\Prob{R = r \given X_0 = x, E=e, B=b} \\
&= \Prob{L=r-b \given X_0 = x, E=e}\\
&= \binom{m-e}{r-j}
      \left[ \frac{\tbp(x)}{1{-}\bp(x)} \right]^{r-j}
      \left[ 1{-}\frac{\tbp(x)}{1{-}\bp(x)} \right]^{(m-e)-(r-j)},
\end{align*}
completing the proof.
\end{proof}

\vspace*{5pt}
\begin{remark}
The sum in Eq.~\eqref{eq:dist-r0-combinatorial-appx} of
Theorem~\ref{thm:dist-r0-combinatorial-appx} converges since $H(x,m,r) \le 1$.
\end{remark}

\vspace*{5pt}
\begin{remark}
Theorem~\ref{thm:dist-r0-combinatorial-appx} shows that it is not the case that
we must have $\bp = \bq$ whenever there exists some $m$ for which the rank
$R$ is uniform on $[m+1]$.  For
example, let $m=1$, let $\dom \defas \set{0, 1, 2, 3}$, let $\prec$ be the usual
order $<$ on $\dom$, and let
$\bp \defas \frac12\delta_0 + \frac12\delta_3$ and
$\bq \defas \frac12\delta_1 + \frac12\delta_2$.
Let $X \sim \bp$ and $Y \sim \bq$. Then we have
$\Prob{R = 0} = \Prob{X > Y} = 1/2 = \Prob{Y < X} = \Prob{R = 1}$.

Rather, Theorem~\ref{thm:uniformity-appx} tells us merely if $R$ is not uniform
on $\set{0,\dots,m}$ for \emph{some} $m$, then $\bp \ne \bq$. In the example
given above, $m=2$ (and so by Theorem~\ref{thm:p-ne-q-finitely-many-m-appx} all
$m \ge 2$) provides such a witness.
\end{remark}

\end{document}